\documentclass[11pt]{amsart}
\usepackage[a4paper,margin=28mm]{geometry}
\usepackage{amsmath,amssymb,mathtools,booktabs,microtype}
\usepackage[colorlinks=true,linkcolor=blue,citecolor=blue,urlcolor=blue]{hyperref}
\newtheorem{theorem}{Theorem}[section]
\newtheorem{proposition}[theorem]{Proposition}
\newtheorem{lemma}[theorem]{Lemma}

\theoremstyle{remark}
\newtheorem{remark}[theorem]{Remark}
\newcommand{\HH}{\mathbb H}
\newcommand{\RR}{\mathbb R}
\newcommand{\ZZ}{\mathbb Z}
\newcommand{\HP}{\mathbb H\mathrm P}
\newcommand{\secg}{\operatorname{sec}}

\newcommand{\II}{\mathrm{II}}
\newcommand{\sign}{\operatorname{sign}}

\newcommand{\eps}{\varepsilon}
\title[Positive curvature on exotic fifteen-spheres]{Positive sectional curvature on a family of exotic fifteen-spheres}
\author{Zhiqi Chen}
\author{Shaoqiang Deng}
\author{Hui Zhang}
\date{Version 2, October 23, 2026}
\makeatletter
\@ifundefined{subjclassname@2020}{%
  \@namedef{subjclassname@2020}{\textup{2020} Mathematics Subject Classification}%
}{}
\makeatother
\subjclass[2020]{Primary 53C20, 57R55; Secondary 55R25, 57R20, 57S25}
\keywords{Exotic sphere, positive sectional curvature, quaternionic Hopf
bundle, gauge group, Eells--Kuiper invariant}
\begin{document}
\begin{abstract}
We construct metrics of strictly positive sectional curvature on at least
508 pairwise nondiffeomorphic exotic fifteen-spheres, without prescribing
orientations. The construction begins with quaternionic line bundles over
$\HP^2\times S^4$. Surgery along a Hopf fiber gives homotopy spheres with
specified equivariant two-disk presentations. A calculation on the surgery
filling distinguishes 1015 nonstandard oriented smooth types. We construct
positively curved quotient metrics on both disks, match their boundary
data under the prescribed attaching map, and smooth the seam by averaging
the inverse tangential metrics. The connection estimates and the smoothing
argument are proved directly for every two-plane. Thus the curvature
conclusion does not require a compatible-disk theorem from the
eleven-dimensional construction or an external curvature-gluing theorem.
We give quantitative parameter bounds and specify the auxiliary choices
for the first nontrivial member after fixing its smooth gauge family.
\end{abstract}
\maketitle
\section{Introduction}

The standard sphere carries a round metric of constant positive sectional
curvature. This familiar example leads to a less elementary question when
one changes the smooth structure. An exotic sphere is a smooth manifold
homeomorphic, but not diffeomorphic, to a standard sphere. A homeomorphism
does not transport a smooth Riemannian metric, so the existence of a
round metric on the standard sphere gives no metric on an exotic one.
It is therefore natural to ask which exotic spheres admit metrics of
positive sectional curvature.

Milnor's construction of exotic seven-spheres~\cite{Milnor} showed that
a sphere can have more than one smooth structure. The subsequent work
of Gromoll and Meyer~\cite{GM} gave an exotic seven-sphere with a metric
of nonnegative sectional curvature. These examples connect two kinds of
questions: how a smooth structure is built, and how curvature behaves
under that construction. For strictly positive curvature, every tangent
two-plane must have positive curvature. A construction which controls
only some distinguished planes does not settle the question.

In this paper we construct a family of exotic fifteen-spheres from
quaternionic line bundles over $\HP^2\times S^4$, followed by surgery
along an $S^{11}$ fiber.
The surgery has a useful geometric interpretation: each resulting
manifold is the union of two standard disks with a specified boundary
diffeomorphism. This suggests constructing a metric on each disk and
joining the two metrics. Two difficulties remain. The boundary metrics
must agree under the very diffeomorphism which records the exotic
smooth structure; moreover, their agreement does not by itself make the
joined metric smooth or preserve positive curvature at the seam.

The two-cap approach in He--Liu--Yau and Deng--Hu--Zhang~\cite{HLY,DHZ}
provides the geometric strategy. Our aim is to carry out this approach
for the quaternionic representation and attaching maps arising from the
fifteen-dimensional surgery construction, with direct proofs of the
geometric estimates needed here. The topological construction, the connection estimate,
and the smoothing argument are presented together so that each
hypothesis can be checked on the same data. The resulting assertion is
as follows.

\begin{theorem}\label{thm:main}
There are at least $508$ pairwise nondiffeomorphic exotic smooth
fifteen-spheres, without prescribed orientation, admitting smooth
Riemannian metrics of strictly positive sectional curvature. More
precisely, the construction gives spheres $\Sigma_m$,
$1\le m\le1015$, with pairwise distinct nonstandard oriented smooth
structures and
\[
 \mu(\Sigma_m)=\frac{15m}{1016}\quad\text{in }\mathbb Q/\ZZ.
\]
Each of these spheres carries such a metric. The subset
$1\le m\le508$ remains pairwise nondiffeomorphic when orientations
are forgotten.
\end{theorem}

Here $\mu$ is the Eells--Kuiper invariant, whose normalization and
calculation are given in Section~\ref{sec:topological-invariant}.
Its role is to distinguish the smooth manifolds; it gives no curvature
information. Conversely, the metric construction applies to each of the
specified disk presentations and does not determine its smooth type.
The two parts of the proof meet at the exact boundary map.

We briefly explain the main steps. A family of gauge transformations
of the Hopf bundle $S^{11}\to\HP^2$ produces a smooth map
$\beta_m:S^{14}\to S^3$ satisfying
\[
 \beta_m(\rho(q)z)=q\beta_m(z)q^{-1},\qquad
 \rho(q)(a,w)=(a,wq^{-1})\quad\text{on }\RR^3\oplus\HH^3.
\]
The surgery sphere is exactly
$D^{15}\cup_{J_{\beta_m}}D^{15}$, where
$J_\beta(z)=\rho(\beta(z))z$. To distinguish these spheres, we use a
sixteen-dimensional filling. Capping a Hopf disk bundle inside that
filling gives a copy of $\HP^3$ with trivial normal bundle. This
submanifold detects the part of the third Pontryagin class which is
not visible on the original bundle base, and completes the invariant
calculation.

For the geometry, the southern disk is a free quotient of a connection
metric. Shrinking the fibers makes vertical curvature large, but the
mixed planes still require attention. A strictly concave variation of
the fiber radius and a four-form correction control the full curvature
expression on simple bivectors. The northern disk uses the radial profile
$F_\delta(s)=s(1+s^2/\delta^2)^{-1/4}$. Its boundary data are chosen to
match the southern disk, while the sum of the outward second fundamental
forms is positive. We then average the inverse tangential metrics in a
normal collar. The derivative jump produces nonnegative curvature terms;
the remaining errors are controlled uniformly on all two-planes.

The paper is organized in this order. Sections
\ref{sec:topological-construction} and \ref{sec:topological-invariant}
construct and distinguish the smooth spheres. Section~\ref{sec:connection}
proves the connection estimate, and Section~\ref{sec:smoothing} proves
the smoothing lemma before it is used. Section~\ref{sec:disks} constructs
the two disks and checks their common boundary. Section~\ref{sec:choices}
specifies the auxiliary functions and parameters for $m=1$.
The quantitative error tables are collected in Appendices
\ref{sec:connection-bounds} and \ref{sec:smoothing-bounds}.
Thus the proof has the following dependencies:
\[
\begin{array}{c}
 \text{gauge family}\longrightarrow
   \text{marked sphere and smooth invariant},\\[2pt]
 \text{connection estimate and radial profile}\longrightarrow
   \text{compatible disks},\\[2pt]
 \text{compatible disks and smoothing}\longrightarrow
   \text{positive curvature on the sphere}.
\end{array}
\]

\subsection*{Relation to earlier work and the first version}
The geometric strategy follows the two-cap constructions
of~\cite{HLY,DHZ}. In particular, the inverse-metric smoothing method
follows~\cite[Section 6.2, version 3]{DHZ}; we prove its needed metric-only
case in Section~\ref{sec:smoothing}. The contribution here is the
fifteen-dimensional family together with the topological and geometric
arguments for its specified attaching maps, rather than a claim that
the two-cap or inverse-metric methods are new.

The first version of this paper used the compatible-disk theorem
of~\cite{DHZ} as the geometric input. In this second version, we prove
the connection estimate, the two quotient-disk constructions, and the
seam-smoothing lemma directly. In particular, we do not invoke a
positive-curvature theorem for eleven-spheres, a compatible-disk theorem,
or the external gluing theorem of Reiser--Wraith~\cite{RW}.
We also make the parameter bounds and the topological interfaces explicit.

The classical inputs are the section-space fibration, the stated
homotopy groups of $S^3$, characteristic-class and index formulas,
and O'Neill's submersion formula. These are identified where they enter.
The construction proves existence for each fixed smooth gauge family;
it does not give a closed formula for a nontrivial family or certified
numerical values of all connection norms. Section~\ref{sec:choices}
makes that distinction explicit.

\subsection*{Conventions}
We use
\[
 R(X,Y)Z=\nabla_X\nabla_YZ-\nabla_Y\nabla_XZ-\nabla_{[X,Y]}Z,
 \qquad
 \secg(X,Y)=\frac{\langle R(X,Y)Y,X\rangle}{|X\wedge Y|^2}.
\]
The unit round sphere has sectional curvature $+1$. On
$\operatorname{Im}\HH$, the Euclidean inner product $Q$ gives
the unit $S^3$ its bi-invariant curvature-one metric; in an oriented
orthonormal basis $[e_a,e_b]=2\epsilon_{abc}e_c$.
For a boundary with outward unit normal $\nu$,
$\II(X,Y)=\langle\nabla_X\nu,Y\rangle$. Thus a Euclidean ball has
positive outward second fundamental form. Cohomology has integral
coefficients unless another coefficient group is written.
\section{Gauge families and the surgery spheres}\label{sec:topological-construction}

The metric construction will use two ordinary disks. Their nonstandard
smooth structure is therefore carried entirely by the map joining their
boundaries. We first construct that map from the Hopf bundle, then show
that it is exactly the attaching map of the surgery construction below. Keeping track of the actual map will also let us identify
the quotient boundary metrics later.

Write quaternionic projective space as the space of right lines and let
$\gamma_r\to\HP^r$ be its tautological quaternionic line bundle. Its unit
sphere bundle is the principal right $S^3$-bundle
\[
 P_r=S(\gamma_r)=S^{4r+3}\longrightarrow\HP^r,
 \qquad w\cdot q=wq.
\]
We use the complex rank-two orientation of each quaternionic line, so that
its Euler class is its second Chern class. Let $\mathcal G_r$ be the gauge
group of $P_r$, based at its identity. A gauge transformation satisfies
$f(wq)=f(w)q$ and covers the identity on $\HP^r$. Choosing a frame above
the unique zero-cell identifies $\mathcal G_0$ with $S^3$.

\begin{lemma}\label{lem:gauge-existence}
For every integer $m$, there is a jointly smooth based family
$f:S^3\to\mathcal G_2$ whose evaluation at the chosen frame has degree
$180m$. The family may be chosen to equal the identity on a neighborhood
of the parameter basepoint.
\end{lemma}

\begin{proof}
Gauge transformations are sections of the adjoint bundle with fiber $S^3$
and conjugation structure group. Write the cellular decomposition
$\HP^r=\HP^{r-1}\cup_\alpha D^{4r}$ and let
$\mathcal A_r=P_r\times_{\mathrm{Ad}}S^3$.
Trivialize the pullback of $\mathcal A_r$ to the characteristic disk.
Restriction to its boundary identifies the continuous section space with
the pullback
\[
 \Gamma(\mathcal A_r)=
 \Gamma(\mathcal A_{r-1})
 \mathop{\times}_{\operatorname{Map}(S^{4r-1},S^3)}
 \operatorname{Map}(D^{4r},S^3).
\]
Here the first map uses the attaching map and the chosen trivialization.
The restriction map
$\operatorname{Map}(D^{4r},S^3)\to
\operatorname{Map}(S^{4r-1},S^3)$ is a fibration because
$S^{4r-1}\hookrightarrow D^{4r}$ is a cofibration.
Its pullback is therefore a fibration on the components under consideration.
The identity section has constant identity boundary value in every
adjoint trivialization; its fiber is exactly
$\operatorname{Map}((D^{4r},S^{4r-1}),(S^3,1))=\Omega^{4r}S^3$.
Equivalently, this is the top row of
\cite[p.~31, diagram~(5.4)]{Gottlieb}, where the double-star spaces
fix the induced base map, as defined on p.~26.
The homotopy exact sequences contain
\[
 \pi_3\mathcal G_1\longrightarrow\pi_3(S^3)
       \longrightarrow\pi_6(S^3),\qquad
 \pi_3\mathcal G_2\longrightarrow\pi_3\mathcal G_1
       \longrightarrow\pi_{10}(S^3).
\]
The classical computations
$\pi_6(S^3)=\ZZ/12$ and $\pi_{10}(S^3)=\ZZ/15$
are recalled in~\cite[Section~1.a and Example~3.6]{CG}.
If $\iota$ generates $\pi_3(S^3)$, exactness gives a class
$a\in\pi_3\mathcal G_1$ evaluating to $12\iota$.
The image of $15a$ in $\pi_{10}(S^3)$ is zero, so $15a$ lifts to
$b\in\pi_3\mathcal G_2$ with evaluation $180\iota$.
A representative of $mb$ gives the required evaluation degree.
Smooth approximation for bundle sections, relative to the basepoint,
gives a jointly smooth representative. Precomposition with a smooth
degree-one map of $S^3$ constant near its basepoint makes the family
the identity there.
For this last approximation one may view $\mathcal A_2$ as the unit
sphere bundle of $\RR\oplus\operatorname{ad}P_2$, approximate the
parameterized section in that vector bundle relative to its identity
neighborhood, and normalize it fiberwise. The approximation preserves
the evaluation degree.
\end{proof}

Put $B=\HP^2\times S^4$. Let $x$ denote the pullback of
$c_2(\gamma_2)$. Orient the parameter sphere so that positive clutching
degree gives positive second Chern number, and let $y$ be its positive
degree-four generator. Thus
\begin{equation}\label{eq:B-cohomology}
 H^*(B;\ZZ)=\ZZ[x,y]/(x^3,y^2),
 \qquad \langle x^2y,[B]\rangle=1.
\end{equation}
For $k=180m$, clutch two copies of $P_2\times D^4$ along their
parameter boundary using a family from Lemma~\ref{lem:gauge-existence}.
Our transition convention is $w_N=f_p(w_S)$.
The resulting principal bundle has an associated quaternionic line bundle
$\xi_k\to B$. Its restriction to $\HP^2\times\{s\}$ is
$\gamma_2$, whereas on $\{b_0\}\times S^4$ it has clutching degree
$k$. Since these restrictions determine degree-four cohomology,
\begin{equation}\label{eq:bundle-euler}
 e(\xi_k)=c_2(\xi_k)=x+ky.
\end{equation}
\begin{remark}[An independent index check of the evaluation degree]
In fact, $180$ is the least positive degree occurring in this construction.
For any quaternionic line bundle on $B=\HP^2\times S^4$ with
$c_2=x+ky$, the spin Dirac operator twisted by that line bundle has index
\[
 \operatorname{ind}D_{\xi_k}
 =\bigl\langle\widehat A(TB)\operatorname{ch}(\xi_k),[B]\bigr\rangle
 =\frac{k}{180}\in\ZZ.
\]
Indeed, $B$ is spin, $\widehat A(TB)=1+x/12$, and, through degree twelve,
\[
 \operatorname{ch}(\xi_k)
 =2-(x+ky)+\frac{(x+ky)^2}{12}-\frac{(x+ky)^3}{360}.
\]
The coefficient of $x^2y$ in their product is
$-k/120+k/72=k/180$. The twisted index theorem~\cite[p.~571, Theorem~(5.3) and Remark~2]{AS}
therefore forces $180\mid k$. Together with
Lemma~\ref{lem:gauge-existence}, this identifies the image of gauge
evaluation with $180\ZZ$. This independent obstruction does not change
the smooth-type count below.
\end{remark}

The sphere bundle $E_k=S(\xi_k)$ is an $S^{11}$-bundle over $S^4$.

Fix the southern clutching chart. At its center, remove the interior of
the product neighborhood $S^{11}\times D^4$ and replace it with
$D^{12}\times S^3$. All parameterizations and normal framings refer to
this fixed product chart. After rounding corners, the result is
\begin{equation}\label{eq:surgery-model}
 \widehat\Sigma_k=(S^{11}\times D^4)\cup_{\Phi_f}(D^{12}\times S^3),
 \qquad \Phi_f(w,p)=(f_p(w),p).
\end{equation}
Equivalently, attach the corresponding index-twelve handle to
$W_k=D(\xi_k)$, and denote the resulting oriented sixteen-manifold by
\begin{equation}\label{eq:filling}
 Z_k=W_k\cup(D^{12}\times D^4),\qquad
 \widehat\Sigma_k=\partial Z_k.
\end{equation}
Here $W_k$ has the product of the orientation in
\eqref{eq:B-cohomology} and the complex fiber orientation.

\begin{lemma}\label{lem:homotopy-sphere}
The manifold $\widehat\Sigma_k$ is a smooth homotopy fifteen-sphere.
\end{lemma}

\begin{proof}
Both pieces in~\eqref{eq:surgery-model} and their intersection
$S^{11}\times S^3$ are simply connected. Each $f_p$ preserves
orientation, because the based continuous family lies in the identity
component of the gauge group. The maps from the intersection to the
pieces on homology in degrees $3$ and $11$ are isomorphisms.
The Mayer--Vietoris sequence consequently gives the integral homology
of $S^{15}$, while van Kampen gives simple connectivity.
The Hurewicz theorem, successively in the first possible nonzero
homotopy degree, and then Whitehead's theorem show that this manifold
is homotopy equivalent to $S^{15}$.
\end{proof}

We next give the disk presentation without using a disk-cutting lemma
from the eleven-dimensional paper. Regard the parameter sphere as
$D^3/\partial D^3$, with $f_p$ the identity near $\partial D^3$.
There is a unique smooth principal division function
\begin{equation}\label{eq:division}
 b(p,w)=\sum_{j=1}^3\overline{w_j}f_{p,j}(w)\in S^3,
 \qquad f_p(w)=wb(p,w).
\end{equation}
Right equivariance implies
\begin{equation}\label{eq:division-equivariance}
 b(p,wq)=q^{-1}b(p,w)q.
\end{equation}
Use the left orthogonal representation
\begin{equation}\label{eq:specific-representation}
 V=\RR^3\oplus\HH^3,
 \qquad \rho(q)(a,w)=(a,wq^{-1}).
\end{equation}
Fix $0<r_0<1$ and define on $S(V)=S^{14}$
\begin{equation}\label{eq:explicit-beta}
 \widehat\beta_k(a,w)=
 \begin{cases}
  b(a/r_0,w/|w|)^{-1},& |a|\leq r_0,\\
  1,&|a|\geq r_0.
 \end{cases}
\end{equation}
On the first region $w\ne0$; the two definitions agree on a collar of
their interface; and the second is constant near $w=0$. Thus the map
is smooth, and~\eqref{eq:division-equivariance} gives
\begin{equation}\label{eq:beta-equivariance}
 \widehat\beta_k(\rho(q)z)=q\widehat\beta_k(z)q^{-1}.
\end{equation}

\begin{proposition}\label{prop:disk-presentation}
Set $J_\beta(z)=\rho(\beta(z))z$. The map
$J_{\widehat\beta_k}$ is an orientation-preserving diffeomorphism and
\begin{equation}\label{eq:disk-presentation}
 \widehat\Sigma_k\cong D^{15}_+\cup_{J_{\widehat\beta_k}}D^{15}_-.
\end{equation}
The disk orientations in this expression are the ones induced from
the oriented surgery manifold, rather than two identical standard
boundary orientations.
\end{proposition}

\begin{proof}
For any conjugation-equivariant $\beta$, substituting $q=\beta(z)$
gives $\beta(J_\beta(z))=\beta(z)$. Hence
$J_\beta^{-1}=J_{\beta^{-1}}$. In our case $J_\beta$ is the identity
on the nonempty open subset $|a|>r_0$, so it preserves orientation.

The elementary equatorial decomposition is
\[
 S^{14}=(S^{11}\times D^3)
             \cup_{S^{11}\times S^2}(D^{12}\times S^2).
\]
The parameterization of its first piece can be fixed explicitly as
\[
 \iota(p,w)=\bigl(r_0p,\sqrt{1-r_0^2|p|^2}\,w\bigr),
 \qquad p\in D^3,\quad w\in S^{11}.
\]
The complementary piece is parameterized by
$(\theta,w)\mapsto(\sqrt{1-|w|^2}\,\theta,w)$, with
$\theta\in S^2$ and $|w|\le\sqrt{1-r_0^2}$.
In these actual coordinates,
\[
 J_{\widehat\beta_k}\bigl(\iota(p,w)\bigr)
       =\iota(p,f_p(w)),
\]
and the map is the identity on the complement. Thus this is an equality
of smooth attaching diffeomorphisms, not an identification merely up to
homotopy.
To see this as the actual surgery transition, write the parameter
$D^4$ in~\eqref{eq:surgery-model} as a rounded
$D^3\times[-1,1]$ and support the family in its top face.
Cut the boundary of the surgery construction across its middle.
The lower half has identity transition. It is the product model
\[
 (S^{11}\times D^3\times[0,1])
       \cup_{S^{11}\times D^3\times\{0\}}(D^{12}\times D^3),
\]
with corners rounded, hence is $D^{12}\times D^3\cong D^{15}$:
the cylindrical piece only adds a boundary collar.
On the cylindrical piece of the upper half make the coordinate change
\[
 (w,p,s)\longmapsto(f_p^{-1}(w),p,s).
\]
It cancels the top-face attachment and is the identity near the side
faces, so the upper half has the same standard disk model. On the
middle cutting face the lower coordinate is $f_p$ applied to the
upper coordinate. On the complementary face it is the identity.
The upper cylinder coordinate and the upper cap coordinate satisfy
$w_{\mathrm{old}}=f_p(w_{\mathrm{cap}})$. Thus
$w_{\mathrm{new}}=f_p^{-1}(w_{\mathrm{old}})$ agrees with the
unchanged cap coordinate, whereas the lower disk retains
$w_{\mathrm{old}}$. This is precisely $J_{\widehat\beta_k}$.
No extension of an individual $f_p$ over $D^{12}$ is asserted or
required: the change of coordinates is made on a cylinder.
The orientations are pulled back from~\eqref{eq:filling}.
\end{proof}

The radial extension $rz\mapsto rJ_{\widehat\beta_k}(z)$ is a
homeomorphism of a closed disk. Applying it to one side identifies the
double in~\eqref{eq:disk-presentation} with the standard topological
$S^{15}$. No smoothness at the center of this radial extension is asserted.

\section{The surgery filling and its smooth invariant}
\label{sec:topological-invariant}

The disk presentation proves that the manifolds are topological spheres,
but it does not yet distinguish their smooth structures. For that purpose
we compute an invariant from the filling $Z_k$. The bundle base detects
the first two Pontryagin classes. In degree twelve a new class appears
because of the handle, so one further submanifold is needed. The capped
Hopf fiber supplies it.

\subsection{The capped Hopf fiber}
The pair $(Z_k,W_k)$ has one relative cell of dimension twelve.
Since $W_k$ retracts onto $B$, restriction induces integral
cohomology isomorphisms in degrees four and eight. Use $x,y$ also
for the unique extensions of these classes to $Z_k$; in particular
$y^2=0$.
The degree-twelve cohomology group of $Z_k$ is free of rank two.
To determine its additional generator geometrically, cap the copy
of $D(\gamma_2)$ over the center of the surgery chart with the
handle core $D^{12}\times\{0\}$, obtaining $F\subset Z_k$.

\begin{lemma}\label{lem:capped-HP3}
The smooth manifold $F$ is diffeomorphic, without a specified
orientation, to $\HP^3$, and its normal bundle in $Z_k$ is the
trivial rank-four bundle. Under this diffeomorphism,
$x|_F=c_2(\gamma_3)$ and $y|_F=0$.
\end{lemma}

\begin{proof}
Represent a point of $D(\gamma_2)$ as $([w],wa)$ with $|w|=1$
and $|a|\leq1$. The map
\[
 ([w],wa)\longmapsto[w,\overline a]\in\HP^3
\]
is well defined. Indeed, replacing $(w,a)$ by $(wq,q^{-1}a)$
replaces $(w,\overline a)$ by $(wq,\overline a q)$, which represents
the same right quaternionic line. Its image is the complement of
the interior of the standard unit ball in the last affine chart.
On $|a|=1$ the affine coordinate is
$w(\overline a)^{-1}=wa$, exactly the boundary parameterization
chosen for the handle. Thus adjoining its standard disk gives
$\HP^3$ with its usual smooth structure.
The four parameter directions of the fixed surgery chart give the
normal framing along $D(\gamma_2)$. The product handle extends
this very framing over the core. More explicitly, the neighborhood is
\[
 (D(\gamma_2)\times D^4_\delta)
 \cup_{S^{11}\times D^4_\delta}(D^{12}\times D^4_\delta)
 \cong F\times D^4_\delta,
\]
with identity transition in the last factor. This proves actual
triviality of the rank-four normal bundle.
The inclusion $\HP^2\subset\HP^3$ induces an isomorphism on
$H^4$, so the restriction of $x$ is $c_2(\gamma_3)$.
The parameter class $y$ vanishes on this capped fiber.
\end{proof}

The class of $F$ maps to a generator of $H_{12}(Z_k,W_k)$,
because its cap is precisely the single handle core. Therefore
restriction to $B$ and $F$ detects degree-twelve integral cohomology:
$x^3$ restricts to zero on $B$ and a primitive top class on $F$,
whereas $x^2y$ restricts to a primitive top class on $B$ and zero
on $F$. Consequently $x^3,x^2y$ are an integral basis of
$H^{12}(Z_k;\ZZ)$.

For the convention $x=c_2(\gamma_r)$, the projective-space
Pontryagin formula is
\begin{equation}\label{eq:HP-Pontryagin}
 p(T\HP^r)=\frac{(1-x)^{2r+2}}{1-4x};
\end{equation}
This follows from the real bundle identity
\[
 T\HP^r\oplus\operatorname{End}_{\HH}(\gamma_r)
       \cong (\gamma_r)^{\oplus(r+1)},
\]
and the formulas $p(\gamma_r)=(1-x)^2$ and
$p(\operatorname{End}_{\HH}(\gamma_r))=1-4x$;
see~\cite[proof of Theorem~3.1]{CG}, with generator $-x$. In particular,
\[
 p(T\HP^2)=1-2x+7x^2,
 \qquad p(T\HP^3)=1-4x+12x^2-8x^3.
\]

\begin{proposition}\label{prop:Pontryagin-filling}
The filling $Z_k$ is spin and its first three Pontryagin classes are
\begin{equation}\label{eq:global-Pontryagin}
 p_1=-4x-2ky,\qquad
 p_2=12x^2+6kxy,\qquad
 p_3=-8x^3-18kx^2y.
\end{equation}
\end{proposition}

\begin{proof}
For a quaternionic line with second Chern class $e$, its underlying
real bundle has $p_1=-2e$ and $p_2=e^2$. Along the zero section
$TZ_k|_B=TB\oplus\xi_k$. Since $TS^4$ is stably trivial,
\eqref{eq:bundle-euler} and~\eqref{eq:HP-Pontryagin} yield
\[
 p(TZ_k)|_B=(1-2x+7x^2)
             \bigl(1-2(x+ky)+(x+ky)^2\bigr).
\]
Using $x^3=y^2=0$ on $B$, the first three terms give the asserted
$p_1,p_2$ and the coefficient $-18k$ of $x^2y$ in $p_3$.
Restrictions determine $p_1,p_2$ uniquely. By
Lemma~\ref{lem:capped-HP3}, $TZ_k|_F=TF\oplus\RR^4$, so the
remaining coefficient of $p_3$ is the coefficient $-8$ of $x^3$
in $p(T\HP^3)$. This proves all three identities.
Finally $H^2(Z_k;\ZZ/2)=0$, so the oriented tangent bundle has
$w_2=0$ and is spin.
\end{proof}

\subsection{Relative products and signature}
The filling has a boundary, so its characteristic numbers must be
interpreted as relative pairings. We spell out those pairings before
using the genus polynomials. This also fixes the orientation signs.

Let $j:H^*(Z_k,\partial Z_k)\to H^*(Z_k)$ be the natural map.
Because the boundary is a homotopy sphere, $j$ is an isomorphism
in degrees $4,8,12$. For any such class $c$, write
$\overline c=j^{-1}(c)$. A degree-sixteen product of positive-degree
classes is evaluated by lifting one factor to relative cohomology
and pairing with $[Z_k,\partial Z_k]$. The result is independent
of the factor chosen, by the compatibility of absolute and relative
cup products. Angle brackets below denote this evaluation.

Let $u=\operatorname{PD}_{Z_k}[B]\in H^4(Z_k,\partial Z_k)$,
with the normal orientation given by $\xi_k$. Its absolute
restriction to $B$ is the Euler class. The degree-four restriction
isomorphism therefore gives
\begin{equation}\label{eq:relative-lifts}
 j(u)=x+ky,\qquad
 \overline{x^2}=u(x-ky),\qquad \overline{xy}=uy.
\end{equation}

\begin{proposition}\label{prop:relative-numbers}
The filling has signature zero and relative Pontryagin numbers
\begin{equation}\label{eq:relative-Pontryagin-numbers}
 \langle p_1^4\rangle=256k,\qquad
 \langle p_1^2p_2\rangle=96k,\qquad
 \langle p_2^2\rangle=0,\qquad
 \langle p_1p_3\rangle=56k.
\end{equation}
\end{proposition}

\begin{proof}
The images of $ux,uy$ under $j$ are $x^2+kxy,xy$, an integral
basis of $H^8(Z_k)$. By the Thom identity their intersection
matrix is
\[
 \begin{pmatrix}
  \langle (x+ky)x^2,[B]\rangle &
  \langle (x+ky)xy,[B]\rangle\\
  \langle (x+ky)xy,[B]\rangle &
  \langle (x+ky)y^2,[B]\rangle
 \end{pmatrix}
 =\begin{pmatrix}k&1\\1&0\end{pmatrix}.
\]
Its determinant is $-1$, so it has one positive and one negative
eigenvalue and hence signature zero.
The same identity and~\eqref{eq:relative-lifts} give
\begin{equation}\label{eq:relative-x-products}
 \langle x^4\rangle=
    \langle (x-ky)x^2,[B]\rangle=-k,\qquad
 \langle x^3y\rangle=
    \langle (x-ky)xy,[B]\rangle=1.
\end{equation}
Products containing $y^2$ vanish. Substitution of
\eqref{eq:global-Pontryagin} gives, before evaluation,
\begin{align*}
 p_1^4&=256x^4+512kx^3y,\\
 p_1^2p_2&=192x^4+288kx^3y,\\
 p_2^2&=144x^4+144kx^3y,\\
 p_1p_3&=32x^4+88kx^3y.
\end{align*}
Here these expressions are used as relative evaluations, rather
than as nonzero degree-sixteen absolute classes on a manifold
with boundary. Applying~\eqref{eq:relative-x-products} proves
\eqref{eq:relative-Pontryagin-numbers}.
\end{proof}

\subsection{The Eells--Kuiper invariant}
We now combine the characteristic numbers into a quantity depending only
on the boundary sphere. The reason for the particular combination below
is that it cancels the top Pontryagin class: the remaining terms admit
the relative evaluations just computed, and changing the filling changes
the answer by an integer index.

We use the following normalization of the Eells--Kuiper
construction~\cite{EK}. The degree-sixteen genus polynomials are
\begin{align}
 \widehat A_4&=
 \frac{381p_1^4-904p_1^2p_2+208p_2^2+512p_1p_3-192p_4}
 {464486400},\label{eq:Ahat16}\\
 L_4&=\frac{-3p_1^4+22p_1^2p_2-19p_2^2-71p_1p_3+381p_4}
 {14175}.\label{eq:L16}
\end{align}
The combination
\begin{equation}\label{eq:N16}
 N_4=\widehat A_4+\frac{L_4}{65024}
 =\frac{153p_1^4-360p_1^2p_2+80p_2^2+192p_1p_3}{187269120}
\end{equation}
has no $p_4$ term and can therefore be evaluated using the relative
products just defined. For a homotopy fifteen-sphere $\Sigma$
bounding a compact spin manifold $Z$, put
\begin{equation}\label{eq:mu-definition}
 \mu(\Sigma)=\langle N_4(TZ),[Z,\partial Z]\rangle
                 -\frac{\sign(Z)}{65024}\quad\text{in }\mathbb Q/\ZZ.
\end{equation}
The original normalization in~\cite[p.~97, (6)--(7)]{EK} divides
by $a_k=4/(3+(-1)^k)$, so $a_4=1$. Moreover,
\cite[p.~99, (9)]{EK} gives $t_4=-1/(2^9\cdot127)=-1/65024$.
Thus no additional factor of two occurs in dimension fifteen.

This value does not depend on the spin filling. To verify this
directly, glue two fillings along their common oriented boundary,
reversing the second orientation. The sphere has a unique spin
structure, so the resulting closed sixteen-manifold $M$ is spin.
Relative characteristic numbers and signature are additive, and
the difference in~\eqref{eq:mu-definition} is
\[
 \langle N_4(TM),[M]\rangle-\frac{\sign(M)}{65024}
       =\widehat A(M)\in\ZZ
\]
by the signature theorem~\cite[p.~577, Theorem~(6.6)]{AS} and the
spin Dirac index theorem~\cite[p.~571, Theorem~(5.3)]{AS}.
The same argument proves oriented diffeomorphism invariance.
Using a disk filling gives $\mu(S^{15})=0$, and orientation
reversal gives $\mu(-\Sigma)=-\mu(\Sigma)$.

\begin{theorem}\label{thm:independent-topology}
For every integer $m$, the preceding construction gives a smooth
conjugation-equivariant map $\beta_m:S^{14}\to S^3$ for the
representation~\eqref{eq:specific-representation} and a homotopy sphere
\[
 \Sigma_m=D^{15}_+\cup_{J_{\beta_m}}D^{15}_-,\qquad
 \mu(\Sigma_m)=\frac{15m}{1016}\pmod{\ZZ}.
\]
The choice for $m=0$ may be made to give the standard sphere.
The $1016$ classes indexed by $0\leq m\leq1015$ have pairwise
different oriented smooth structures. All except $m=0$ are exotic,
and the $508$ spheres indexed by $1\leq m\leq508$ remain pairwise
nondiffeomorphic after forgetting orientation.
\end{theorem}

\begin{proof}
Set $k=180m$, $\Sigma_m=\widehat\Sigma_k$, and
$\beta_m=\widehat\beta_k$. The disk presentation is
Proposition~\ref{prop:disk-presentation}. By
Proposition~\ref{prop:relative-numbers} and~\eqref{eq:N16},
\[
 \mu(\widehat\Sigma_k)
 =\frac{153(256k)-360(96k)+192(56k)}{187269120}
 =\frac{k}{12192}.
\]
Since $180/12192=15/1016$, this is the asserted value.
For $m=0$ use the identity family, whose disk twist is the identity.
As $\gcd(15,1016)=1$, the displayed invariant values are distinct
for $0\leq m\leq1015$, and the last $1015$ values are nonzero.
Their spheres cannot be diffeomorphic to the standard sphere.
An unoriented diffeomorphism either preserves or reverses
orientation, so a diffeomorphism between $\Sigma_i$ and $\Sigma_j$
would imply $i\equiv j$ or $i\equiv-j\pmod{1016}$.
For $1\leq i<j\leq508$, neither congruence is possible.
No assertion that $\mu$ classifies all homotopy fifteen-spheres
is needed.
\end{proof}

\begin{remark}\label{rem:topology-geometric-separation}
We have now obtained the two pieces of topological information needed
later: an exact disk presentation and an invariant distinguishing its
smooth type. The remaining question is geometric. We must put positive
curvature on this particular gluing, beginning with the connection
estimate in the next section.
\end{remark}
\section{Connection metrics with small fibers}
\label{sec:independent-connection}
\label{sec:connection}

A connection splits the tangent bundle of a principal bundle into
horizontal and vertical directions. If the base is positively curved,
shrinking the fibers preserves positivity on horizontal planes and
increases the curvature of vertical planes. The mixed planes are the
main difficulty. We will vary the fiber radius by a strictly concave
function, then estimate all six curvature blocks together. The result
applies up to the boundary of a compact base and will provide the
southern disk in Section~\ref{sec:disks}.

Let $Q$ be the bi-invariant metric on $\operatorname{Sp}(1)$ of constant
sectional curvature one, so that $[u,v]=2u\times v$ in an oriented
$Q$-orthonormal basis of its Lie algebra.  Our curvature convention is
\[
 R(A,B,C,D)=\langle
 (\nabla_A\nabla_B-\nabla_B\nabla_A-\nabla_{[A,B]})D,C\rangle.
\]
Thus $R(A,B,A,B)$ is positive on a round sphere.

\begin{proposition}[Small fibres with a strictly concave radius]
\label{prop:independent-connection}
Let $(B,g)$ be compact, possibly with smooth boundary, with
$\sec_g\geq\kappa>0$.  Let $\pi:P\longrightarrow B$ be a smooth principal
$\operatorname{Sp}(1)$-bundle, and let $\omega$ be any smooth connection,
with curvature $\Omega$.  All these data are smooth up to the boundary.
The following two tensors define finite constants depending only on
$(g,\omega)$:
\begin{align*}
 S(X\otimes u,Y\otimes v)
  &=\frac14\sum_i
   Q(\Omega(X,e_i),v)Q(\Omega(Y,e_i),u),\\
 L(X\wedge Y,Z\otimes u)
  &=-\frac12Q((D_Z\Omega)(X,Y),u),\\
 C_0&=\sup_B\|S\|_{\mathrm{op}},\qquad
 C_1=\sup_B\|L\|_{\mathrm{op}}.
\end{align*}
Here $D$ is the covariant derivative induced by $g$ and $\omega$,
$\{e_i\}$ is any $g$-orthonormal frame, and operator norms use $g$ and $Q$.
Suppose a smooth function $\phi$ on $B$ satisfies
\begin{equation}
\label{eq:independent-concavity}
 -\tfrac12\operatorname{Hess}_g\phi
 \ \geq\left(C_0+\frac{2C_1^2}{\kappa}+1\right)g.
\end{equation}
Then there exists $\varepsilon_0>0$ such that for every
$0<\varepsilon<\varepsilon_0$ the metric
\begin{equation}
\label{eq:independent-connection-metric}
 G_\varepsilon=\pi^*g+
 r_\varepsilon^2 Q(\omega,\omega),\qquad
 r_\varepsilon^2=\varepsilon\exp(\varepsilon\phi\circ\pi),
\end{equation}
has strictly positive sectional curvature at every point of $P$, including
its boundary.  The threshold is uniform on $P$.
\end{proposition}

\begin{proof}
We give the curvature calculation and the estimate for arbitrary planes.
Until the final scaling argument, $r$ is any positive smooth basic function.
Set $a=d\log r$ and write $a^\sharp=\operatorname{grad}\log r$.
Horizontal vectors are identified with vectors on $B$.  For
$u\in\mathfrak{sp}(1)$, write $U=r^{-1}u^*$, where $u^*$ is the
fundamental field for the right principal action.  Consequently
$\langle U,V\rangle=Q(u,v)$.

For local horizontal lifts $X,Y$ and constant Lie algebra elements $u,v$,
the brackets are
\[
 [X,Y]=[X,Y]_B^{\mathrm H}-\Omega(X,Y)^*,\qquad
 [X,U]=-a(X)U,\qquad
 [U,V]=r^{-1}[u,v]^{\mathrm{unit}}.
\]
The last superscript means the normalized vertical identification
$w^{\mathrm{unit}}=r^{-1}w^*$.  An essential additional identity is
\begin{equation}
\label{eq:vertical-Omega-derivative}
 U\bigl(\Omega(X,Y)\bigr)=-r^{-1}[u,\Omega(X,Y)].
\end{equation}
It follows from the equivariance of curvature; its omission would change
the leading mixed curvature terms.

Define $F_{X,u}\in TB$ by
$g(F_{X,u},Y)=Q(\Omega(X,Y),u)$.  Koszul's formula gives
\begin{align}
\label{eq:connection-koszul}
 \nabla_XY&=(\nabla^B_XY)^{\mathrm H}
                 -\tfrac r2\Omega(X,Y)^{\mathrm{unit}},\nonumber\\
 \nabla_XU&=\tfrac r2F_{X,u}^{\mathrm H},\nonumber\\
 \nabla_UX&=\tfrac r2F_{X,u}^{\mathrm H}+a(X)U,\\
 \nabla_UV&=\tfrac1{2r}[u,v]^{\mathrm{unit}}
                  -Q(u,v)a^\sharp.\nonumber
\end{align}
These formulas apply to the stated local extensions; derivatives of
nonconstant coefficients must still be taken.  Substitution into the
definition of $R$, using \eqref{eq:vertical-Omega-derivative} and
$D\Omega=0$ in the exterior covariant sense, yields the following complete
list of curvature blocks.  In the formulas, $X,Y,Z,W$ are horizontal and
$U,V,W_v,Z_v$ correspond respectively to $u,v,w,z$ in the normalized
vertical identification:
\begin{align}
\label{eq:connection-HHHH}
 R(X,Y,Z,W)&=R_B(X,Y,Z,W)\\
 &\quad-\frac{r^2}{4}\bigl(2Q(\Omega(X,Y),\Omega(Z,W))\nonumber\\
 &\qquad\qquad+Q(\Omega(X,Z),\Omega(Y,W))
              -Q(\Omega(X,W),\Omega(Y,Z))\bigr),\nonumber\displaybreak[1]\\
\label{eq:connection-HHHV}
 R(X,Y,Z,U)&=-\frac r2Q((D_Z\Omega)(X,Y),u)\\
 &\quad+\frac r2\bigl(a(X)Q(\Omega(Y,Z),u)
                   -a(Y)Q(\Omega(X,Z),u)\bigr)\nonumber\\
 &\quad-r a(Z)Q(\Omega(X,Y),u),\nonumber\displaybreak[1]\\
\label{eq:connection-HHVV}
 R(X,Y,U,V)&=-\frac12Q(\Omega(X,Y),[u,v])\\
 &\quad+\frac{r^2}{4}\sum_i\bigl(
 Q(\Omega(X,e_i),v)Q(\Omega(Y,e_i),u)\nonumber\\
 &\hspace{44mm}-Q(\Omega(X,e_i),u)Q(\Omega(Y,e_i),v)\bigr),\nonumber\displaybreak[1]\\
\label{eq:connection-HVHV}
 R(X,U,Y,V)&=-\frac14Q(\Omega(X,Y),[u,v])
                 +r^2 S(X\otimes u,Y\otimes v)\\
 &\quad-Q(u,v)\bigl(\operatorname{Hess}\log r(X,Y)
                         +a(X)a(Y)\bigr),\nonumber\displaybreak[1]\\
\label{eq:connection-HVVV}
 R(X,U,V,W_v)&=\frac r2\bigl(
 Q(u,w)Q(\Omega(X,a^\sharp),v)
 -Q(u,v)Q(\Omega(X,a^\sharp),w)\bigr),\\
\label{eq:connection-VVVV}
 R(U,V,W_v,Z_v)&=(r^{-2}-|a|^2)
       \bigl(Q(u,w)Q(v,z)-Q(u,z)Q(v,w)\bigr).
\end{align}
For clarity, two of the terms in this computation can be checked directly.
In $R(X,U,Y,V)$ the horizontal component of
$\nabla_X\nabla_UV$ contributes
$\tfrac14Q(\Omega(X,Y),[u,v])$; differentiating the coefficient
of $\nabla_XV$ in the $U$ direction contributes
$\tfrac12Q(\Omega(X,Y),[u,v])$ with the opposite sign.
This gives the coefficient $-\tfrac14$ in
\eqref{eq:connection-HVHV}.  In \eqref{eq:connection-HVVV}, the
terms $\mp a(X)[u,v]/(2r)$ cancel.  Thus this block is of order
$r|a||\Omega|$, with no term of order $|a|/r$.

The term involving the Lie bracket in
\eqref{eq:connection-HVHV} has no small factor.  Bounding it by a constant
times the square of an arbitrary mixed bivector would not prove the
proposition.  Instead define a four-form $\eta_r$ by its only nonzero
horizontal--vertical type:
\begin{equation}
\label{eq:connection-four-form}
 \eta_r(X,Y,U,V)=-\tfrac14Q(\Omega(X,Y),[u,v]).
\end{equation}
Extend this by total alternation, and set its types other than $(2,2)$
equal to zero.  This is a globally defined smooth four-form: the defining
expression is invariant under the simultaneous adjoint action.
Let $\widehat R=R+\eta_r$, regarded as a symmetric bilinear form on
$\Lambda^2TP$.  Since
$\eta_r(A,B,A,B)=0$, one has
\begin{equation}
\label{eq:connection-same-sectional}
 \widehat R(A\wedge B,A\wedge B)=R(A,B,A,B).
\end{equation}
On the mixed block,
$\eta_r(X,U,Y,V)=+\tfrac14Q(\Omega(X,Y),[u,v])$,
which cancels exactly the first term of
\eqref{eq:connection-HVHV}.  The horizontal--vertical two-form block
\eqref{eq:connection-HHVV} remains bounded (its bracket coefficient
becomes $-3/4$).  This is the algebraic use of decomposability required
in the proof.

Now put $r^2=\varepsilon e^{\varepsilon\phi}$.  Then
\[
 a=\tfrac\varepsilon2d\phi,\qquad
 \operatorname{Hess}\log r=\tfrac\varepsilon2
                                      \operatorname{Hess}\phi.
\]
Use the orthogonal splitting
\[
 \Lambda^2TP=\Lambda^2H\ \oplus\ (H\otimes V)\ \oplus\ \Lambda^2V.
\]
Identifications with $TB$ and $\mathfrak{sp}(1)$ are the normalized ones
already fixed.  The matrices of $\widehat R$ after the congruence
\[
 T_\varepsilon=\operatorname{diag}
                 (1,\varepsilon^{-1/2},\varepsilon^{1/2})
\]
converge uniformly, in operator norm, to
\begin{equation}
\label{eq:connection-limit}
 T_\varepsilon^*\widehat R T_\varepsilon
 \ \longrightarrow
 \begin{pmatrix}
   R_B&L&0\\
   L^*&-\tfrac12\operatorname{Hess}\phi\otimes Q+S&0\\
   0&0&I
 \end{pmatrix}.
\end{equation}
This is uniform on the compact base and its boundary.  More explicitly,
the errors in the diagonal horizontal, mixed, and vertical blocks are
$O(\varepsilon)$; the horizontal--mixed error is $O(\varepsilon)$;
the horizontal--vertical block is $O(\sqrt\varepsilon)$; and the
mixed--vertical block is $O(\varepsilon^{3/2})$.
All constants are finite combinations of the uniform norms of
$\Omega,D\Omega,\phi,d\phi,\operatorname{Hess}\phi$ and $R_B$.
These claims follow directly from
\eqref{eq:connection-HHHH}--\eqref{eq:connection-VVVV}; no uncomputed
curvature block is included in the error.

The curvature operator $R_B$ need not be positive definite on all
bivectors.  Accordingly we only apply the limiting estimate to a simple
bivector $\xi=A\wedge B$.  Write
$\xi=\xi_H+\xi_M+\xi_V$, and put
\[
 h=|\xi_H|,\qquad M=\sqrt\varepsilon|\xi_M|,
 \qquad V=\varepsilon^{-1/2}|\xi_V|.
\]
More explicitly, evaluate the congruent form on
\(w=T_\varepsilon^{-1}\xi\), so
\(|w|^2=h^2+M^2+V^2\) and its horizontal component is still \(\xi_H\).
The horizontal part is itself simple,
$\xi_H=A_H\wedge B_H$.  Thus
$R_B(\xi_H,\xi_H)\geq\kappa h^2$.  By
\eqref{eq:independent-concavity}, the limiting quadratic form in
\eqref{eq:connection-limit} is at least
\begin{align*}
 \kappa h^2-2C_1hM+
           (2C_1^2/\kappa+1)M^2+V^2
 &\geq\tfrac\kappa2h^2+M^2+V^2.
\end{align*}
Uniform convergence implies that, after reducing $\varepsilon_0$, the
actual form is at least
\[
 \tfrac12\min\{\kappa/2,1\}\,(h^2+M^2+V^2)>0
\]
for every nonzero simple $\xi$ and every
$0<\varepsilon<\varepsilon_0$.  Equation
\eqref{eq:connection-same-sectional} finishes the proof.
\end{proof}

\begin{remark}
The numerical constants in
\eqref{eq:independent-concavity} are deliberately conservative.  If a
fixed smooth function $\psi$ satisfies
$-\operatorname{Hess}\psi\geq c g$ for some $c>0$, multiplying $\psi$
by a sufficiently large fixed constant supplies $\phi$.
The choice of that constant precedes, and is independent of, the choice
of sufficiently small $\varepsilon$.  Merely semidefinite concavity at a
boundary point does not meet the hypothesis.
\end{remark}

For use in the construction, Appendix~\ref{sec:connection-bounds}
records an explicit threshold $\varepsilon_*$ and the lower bound
$\sec_{G_\varepsilon}\ge\tfrac12\min\{\kappa/2,1\}\varepsilon$.
Its table follows from the curvature blocks just proved.
\section{A direct smoothing of the seam}\label{sec:smoothing}

Two smooth metrics with the same boundary metric can be joined to form
a continuous metric, but their normal derivatives need not agree. The
resulting corner is measured by the sum of the outward second fundamental
forms. We show that a nonnegative sum allows the corner to be smoothed
while retaining a positive sectional-curvature lower bound.

The useful quantity to average is the inverse tangential metric. For this
choice, the jump in the first derivative contributes a nonnegative term
to radial curvature. The proof must also control tangential and mixed
planes, as well as the region where the averaging is cut off. We give
these estimates below, following the inverse-metric method
of~\cite[Section 6.2, version 3]{DHZ}.

\begin{lemma}[Harmonic smoothing]\label{lem:harmonic}
Let two compact smooth Riemannian manifolds have isometric boundaries
and sectional curvatures at least \(\kappa>0\), including at their
boundaries. Identify those boundaries by the prescribed isometry.
Suppose the sum of their outward second fundamental forms is
nonnegative. Then, inside any prescribed boundary collar, the joined
continuous metric can be replaced by a smooth metric with sectional
curvature greater than \(\kappa/2\). It agrees exactly with the original
metrics outside that collar. The replacements converge uniformly to
the joined metric and smoothly on compact sets away from the seam.
\end{lemma}

\begin{proof}
Choose normal geodesic collars on the two sides. After pulling back by
the boundary isometry, their common signed collar is
\((-c,c)\times Y\), and the metric has the form
\[
 g=dt^2+h(t),\qquad h(0^-)=h(0^+)=h_0.
\]
The functions \(h_-\) and \(h_+\) are smooth up to zero from their
respective sides. The outward normals are \(\partial_t\) on the
negative side and \(-\partial_t\) on the positive side. Put
\begin{equation}\label{eq:smooth-jump}
 P_\pm=h'_\pm(0),\qquad Q_\pm=h''_\pm(0),\qquad
 D=P_--P_+=2(\II_-+\II_+)\ge0.
\end{equation}
All tangential tensors at different normal parameters are identified
using the product collar. Fix the connection \(\nabla^0\) of \(h_0\)
in the tangential variables. In the calculations at a point of \(Y\),
use an \(h_0\)-orthonormal frame and matrix notation, so that \(h_0=I\).
Products of covariant tensors below mean contraction with \(h_0^{-1}\).
This is only a notation for tensorial identities, not an assumption
that \(h_0\) is flat.

\emph{The smoothing operation.}
Take a nonnegative even function \(\varrho\in C_c^\infty((-1,1))\)
with integral one. Choose a fixed smooth function \(0\le\theta\le1\)
equal to one on \(|t|\le c/3\) and zero on \(|t|\ge2c/3\).
For \(0<\zeta<c/12\), let
\begin{equation}\label{eq:harmonic-definition}
 \varrho_\zeta(t)=\zeta^{-1}\varrho(t/\zeta),\quad p(t)=h(t)^{-1},
 \quad
 p_\zeta=\theta(\varrho_\zeta*p)+(1-\theta)p,\quad
 g_\zeta=dt^2+H,\quad H=p_\zeta^{-1}.
\end{equation}
The convolution is taken only in \(t\). It is needed only where
\(|t|<2c/3\), so its arguments remain in the original collar.
It averages positive contravariant tensors in the same vector bundle;
thus \(p_\zeta\) is positive definite. The formula is smooth across
zero, agrees with the original metric near the ends, and defines a
global smooth metric. Tangential coordinate changes are independent
of \(t\), so the formula does not depend on the chosen charts.
Uniform convergence follows from the approximate-identity property
and uniform positivity of \(p\). Smooth convergence away from zero
follows by differentiating on one smooth side.

\emph{The complete collar curvature formulas.}
For a smooth family \(H(t)\), write \(R_H\) for its intrinsic
\((0,4)\) curvature with the convention
\(R(A,B,C,D)=\langle R(A,B)D,C\rangle\). The Christoffel symbols
in product coordinates are
\[
 \Gamma^t_{ij}=-\tfrac12H'_{ij},\qquad
 \Gamma^i_{tj}=\tfrac12H^{ik}H'_{kj},\qquad
 \Gamma^i_{jk}=\Gamma^i_{jk}(H),
 \qquad \Gamma^t_{tt}=\Gamma^i_{tt}=\Gamma^t_{ti}=0.
\]
Substitution into the definition of curvature gives its three blocks:
\begin{align}
 \mathcal A_H(u,v)&=R(\partial_t,u,\partial_t,v)
       =-\tfrac12H''(u,v)+\tfrac14(H'H^{-1}H')(u,v),
       \label{eq:smooth-radial}\\
 \mathcal B_H(u\wedge v,u\wedge v)
       &=R_H(u,v,u,v)
         -\tfrac14\{H'(u,u)H'(v,v)-H'(u,v)^2\},
         \label{eq:smooth-tangent}\\
 \mathcal C_H(u,v\wedge w)&=R(\partial_t,u,v,w)
       =\tfrac12\{(\nabla^H_wH')(u,v)
                         -(\nabla^H_vH')(u,w)\}.
         \label{eq:smooth-codazzi}
\end{align}
For example, the angular formula follows from the tangential
Christoffel symbols and the two terms containing \(\Gamma^t_{ij}\);
the radial formula includes the derivative of \(H^{-1}\) in
\(\Gamma^i_{tj}\). These formulas, their polarizations, and the
curvature symmetries give all components. In particular the Codazzi
block is included rather than inferred from the other two blocks.

\emph{Uniform expansions in the seam region.}
For \(|t|\le\zeta\), define
\[
 w=\int_0^\infty\varrho_\zeta(t-s)\,ds,\qquad
 J=\varrho_\zeta(t),\qquad v=w(1-w),
 \qquad \overline{T}=(1-w)T_-+wT_+.
\]
Here \(0\le w\le1\), \(J\ge0\), and \(w\) depends only on \(t\).
From \(p'=-p h'p\) and
\(p''=2p h'p h'p-p h''p\), the one-sided values at zero are
\[
 p(0)=I,\qquad p'_\pm(0)=-P_\pm,\qquad
 p''_\pm(0)=2P_\pm^2-Q_\pm.
\]
Since \(p\) is continuous, its distributional first derivative has
no atom. Its distributional second derivative is its ordinary
piecewise second derivative plus \(D\delta_0\).
Taylor's theorem on each half-collar therefore gives
\begin{align}
 p_\zeta&=I+E_0,&
 p'_\zeta&=-\overline P+E_1,\nonumber\\
 p''_\zeta&=2\overline{P^2}-\overline Q+JD+E_2,
 &\|E_0\|_{C^2(Y)}+\|E_1\|_{C^1(Y)}+\|E_2\|_{C^0(Y)}
 &\le C_0\zeta.\label{eq:smooth-taylor}
\end{align}
The norms use \(h_0,\nabla^0\). The constant \(C_0\) is independent
of \(t,y,\zeta\): it can be chosen from the suprema of the one-sided
derivatives \((\nabla^0)^i\partial_t^j p\), \(i+j\le3\),
on the compact closed half-collars. Every sampled point has normal
distance at most \(2\zeta\) from zero. Thus these estimates hold also
at \(t=\pm\zeta\); no lower bound on \(J\) is required.

Inverting the first estimate and differentiating gives
\(H=I+O_{C^2(Y)}(\zeta)\) and
\(H'=\overline P+O_{C^1(Y)}(\zeta)\). In the second derivative we
must keep the distributional term exactly:
\begin{equation}\label{eq:smooth-second}
 H''=2Hp'_\zeta Hp'_\zeta H
       -H(2\overline{P^2}-\overline Q+E_2)H-JHDH
     =\overline Q-2vD^2-JHDH+O(\zeta).
\end{equation}
The identity
\(\overline{P^2}-(\overline P)^2=vD^2\) holds without a
commutativity assumption. All terms in the regular part are bounded,
so replacing their factors by their limiting values costs
\(O(\zeta)\). The possibly unbounded \(JHDH\) is not approximated:
replacing it by \(JD\) would in general create an order-one error.
In intrinsic notation this exact tensor is
\(JHh_0^{-1}Dh_0^{-1}H\), which is nonnegative.

Let \(\mathcal A_\pm,\mathcal B_\pm,\mathcal C_\pm\) denote
the curvature blocks of the original metrics at \(0^\pm\), in the
common reference metric \(g_0=dt^2+h_0\). Equations
\eqref{eq:smooth-radial}--\eqref{eq:smooth-second} yield
\begin{align}
 \mathcal A_H&=\overline{\mathcal A}
                    +\tfrac34vD^2+\tfrac12JHDH+O(\zeta),
                    \label{eq:smooth-blocks}\\
 \mathcal B_H(b,b)&=\overline{\mathcal B}(b,b)
                    +\tfrac14v(\Lambda^2D)(b,b)
                    +O(\zeta)|b|_0^2,\nonumber\\
 \mathcal C_H&=\overline{\mathcal C}+O(\zeta).\nonumber
\end{align}
For completeness, the radial coefficient is obtained by subtracting
\(-\overline Q/2+\overline{P^2}/4\) from
\(-\overline Q/2+vD^2+(\overline P)^2/4\).
For the tangential coefficient, define
\(\Delta_T(u,v)=T(u,u)T(v,v)-T(u,v)^2\).
Then \(\overline{\Delta_P}-\Delta_{\overline P}=v\Delta_D\).
Here \((\Lambda^2D)(u\wedge v,z\wedge w)
=D(u,z)D(v,w)-D(u,w)D(v,z)\), so its value on
\(u\wedge v\) is exactly \(\Delta_D(u,v)\).
Also \(R_H=R_{h_0}+O(\zeta)\), since \(H-h_0=O_{C^2(Y)}(\zeta)\).
For the mixed block, \(\nabla^H-\nabla^0=O(\zeta)\) and
\(H'-\overline P=O_{C^1(Y)}(\zeta)\); differentiating tangentially
does not differentiate \(w\) or \(J\). This proves its stated
uniform error directly. These observations give finite constants
for every error in \eqref{eq:smooth-blocks}, depending only on the
fixed collar data. In particular none contains a hidden factor \(J\).

\emph{Every two-plane, including tilted planes.}
A simple bivector has the decomposition
\[
 \xi=(a\partial_t+u)\wedge(b\partial_t+z)
       =\partial_t\wedge A+B,\qquad
 A=az-bu,\quad B=u\wedge z,
 \qquad |\xi|_0^2=|A|_0^2+|B|_0^2.
\]
Both endpoint curvature forms evaluate the same simple \(\xi\),
so each is at least \(\kappa|\xi|_0^2\). The additional radial
terms in \eqref{eq:smooth-blocks} are nonnegative. Its tangential
term is nonnegative on \(B\), because \(D\ge0\) implies
\(D(u,u)D(z,z)-D(u,z)^2\ge0\). Including the mixed error gives,
for a finite constant \(C_1\),
\[
 R^{g_\zeta}(\xi,\xi)\ge
 \kappa|\xi|_0^2-C_1\zeta
       (|A|_0^2+2|A|_0|B|_0+|B|_0^2)
 \ge(\kappa-2C_1\zeta)|\xi|_0^2.
\]
Moreover \(|\xi|_{g_\zeta}^2\le(1+C_2\zeta)^2|\xi|_0^2\)
for a finite \(C_2\). Consequently
\begin{equation}\label{eq:smooth-lower}
 \secg_{g_\zeta}\ge
 \frac{\kappa-2C_1\zeta}{(1+C_2\zeta)^2}
 \quad\text{on }|t|\le\zeta.
\end{equation}
This argument uses sectional positivity only on simple bivectors;
it does not assume positivity of either endpoint curvature operator.

\emph{The remainder of the collar and exact matching.}
If \(|t|\ge\zeta\), the convolution samples just one smooth side.
Taylor's theorem applied to derivatives through order two, with
the one-sided third derivatives bounded up to zero, gives
\[
 \|p_\zeta-p\|_{C^2}\le C_3\zeta
       \quad (\zeta\le|t|\le2c/3).
\]
The fixed cutoff and its first two derivatives only change \(C_3\).
This estimate is uniform even as \(t\) approaches \(\pm\zeta\).
All inverse tensors stay in a uniformly positive set. The formulas
for inversion and their first two derivatives thus give
\(\|g_\zeta-g\|_{C^2}\le C_4\zeta\) on this region.
In a finite atlas on the closed half-collars, the coordinate formula
for curvature is a smooth function of \(g,g^{-1},\partial g,
\partial^2g\). Its derivatives are bounded on the compact set of
jets in question. It follows that the curvature numerator and the
squared area denominator each differ by at most \(C_5\zeta\) on
reference unit simple bivectors. Uniform metric equivalence then
gives \(\secg_{g_\zeta}\ge\kappa-C_6\zeta\) there.
This supplies the needed uniform curvature-openness estimate,
including the cutoff region, from the actual jets.

Increase the constants to be at least one. It suffices to choose
\[
 0<\zeta<\min\{c/12,\ \kappa/(8C_1),\ 1/(8C_2),
                         \kappa/(2C_6)\},
\]
as well as the fixed smallness threshold used to obtain uniform
invertibility. Equation~\eqref{eq:smooth-lower} is then greater
than \(\kappa/2\), and so is the bound on the rest of the collar.
Outside the support of \(\theta\), the metric is exactly the original
one, with all derivatives matching. All constants are fixed before
\(\zeta\) is chosen. This completes the proof.
\end{proof}

Appendix~\ref{sec:smoothing-bounds} gives a quantitative version of the
error estimates. Its constants depend only on fixed one-sided metric
jets, the chosen collar, and the original curvature lower bound. They
can therefore be fixed before the smoothing width is chosen.
\section{A direct construction of the two quotient disks}\label{sec:disks}

We now apply the two analytic results to the boundary maps constructed
in Section~\ref{sec:topological-construction}. A principal bundle carries
a second, free group action, and each cap of its quotient is a standard
disk. The bundle transition will give exactly the required boundary map.
We then choose the two cap metrics so that their boundaries agree and
the smoothing lemma applies.

Throughout this section
\[
 V=\RR^3\oplus\HH^3,\qquad
 \rho(q)(a,w)=(a,wq^{-1}),\qquad \beta:S(V)\longrightarrow S^3
\]
is a smooth map satisfying
\(\beta(\rho(q)z)=q\beta(z)q^{-1}\). Put
\(J_\beta(z)=\rho(\beta(z))z\). We keep the boundary parametrizations throughout the construction,
since they determine which smooth sphere the metric will belong to.

\subsection{The bundle, its actions, and the disk coordinates}
Write \(S^{15}\) in suspension coordinates \((t,z)\),
\(0\le t\le\pi\), \(z\in S(V)\). Set \(g(z)=\beta(z)^{-1}\) and
glue two trivial principal right \(S^3\)-bundles by
\begin{equation}\label{eq:transition}
 u_S=g(z)u_N.
\end{equation}
The transition is independent of \(t\). In both trivializations define
\begin{equation}\label{eq:star}
 q\star(t,z,u)=(t,\rho(q)z,qu).
\end{equation}
Conjugation equivariance of \(g\) makes this action well defined in the
overlap. It commutes with the principal right action and is free, including
over the poles: \(qu=u\) implies \(q=1\).

If \(\zeta\in D(V)\) is an extending disk coordinate on either cap, then
\(\rho(u^{-1})\zeta\) is an invariant quotient coordinate. Every orbit
has a unique representative with \(u=1\); hence each quotient cap is a
standard smooth disk, even though the representation on \(V\) has fixed
vectors. If \(z_N=\rho(u_N^{-1})z\), then on their common boundary
\begin{align*}
 z_S&=\rho(u_S^{-1})z
     =\rho\bigl(u_N^{-1}\beta(z)u_N\bigr)z_N\\
    &=\rho\bigl(\beta(z_N)\bigr)z_N=J_\beta(z_N).
\end{align*}
Thus the quotient gluing is exactly \(D^{15}_N\cup_{J_\beta}D^{15}_S\).

\subsection{A southern connection metric}
The connection estimate requires a strictly concave radius potential.
A round cap smaller than a hemisphere provides such a function. We
choose its size first, then fix a connection which is a product near
the cutting boundary and smooth near the pole.

Fix once and for all
\[
 a=2\pi/3,\qquad b=\sin a=\sqrt3/2.
\]
The southern cap is \(a\le t\le\pi\) with round metric
\(h_S=dt^2+\sin^2t\,h_{S(V)}\). Its radius from the south pole is
\(\pi/3\), less than a hemisphere. Choose a smooth radial function
\(\chi\) equal to zero on a collar of \(a\) and equal to one near
\(\pi\). Use the connection potentials
\begin{equation}\label{eq:potentials}
 A_N=\chi g^{-1}dg,\qquad A_S=(\chi-1)dg\,g^{-1}.
\end{equation}
They satisfy \(A_N=g^{-1}A_Sg+g^{-1}dg\), the transition law dictated
by \eqref{eq:transition}. The southern potential vanishes near its pole,
so it extends smoothly there. The north potential is only being used on
the overlap and vanishes near the cutting boundary. Both potentials
transform by \(\rho(q)^*A=qAq^{-1}\). Consequently
\(\omega=u^{-1}Au+u^{-1}du\) is invariant under \eqref{eq:star}.

The connection and its curvature bounds are now fixed, before any small
fiber parameter is chosen. Apply the connection-metric criterion of
Proposition~\ref{prop:independent-connection} to choose, explicitly,
\(A=4(C_0+2C_1^2+1)\), with its constants computed on this fixed round
cap and connection. Then \(\varphi(t)=-A\cos t\) meets that criterion.
Indeed
\[
 -\operatorname{Hess}_{h_S}\varphi=-A\cos t\,h_S
 \ \ge\ (A/2)h_S.
\]
For all sufficiently small \(\eps>0\), the metric
\begin{equation}\label{eq:southmetric}
 G_S=h_S+r_S^2Q(\omega,\omega),\qquad
 r_S^2=\eps e^{-\eps A\cos t}
\end{equation}
has positive sectional curvature on every source two-plane. The source
metric is smooth at the pole: the connection is zero there and
\(r_S(\pi-s)\) is smooth and even in \(s\). Its free isometric
\(\star\)-quotient is therefore a smooth positively curved disk by
the Riemannian submersion curvature formula.

Write its boundary data as
\begin{equation}\label{eq:southdata}
 r_a^2=\eps e^{\eps A/2},\qquad
 q=\frac{r'_S(a)}{r_a}=\frac{\eps A b}{2},\qquad
 \mu_S=\cot a=-1/\sqrt3.
\end{equation}
In the common north gauge the source metric is a doubly warped product
on a collar. For normalized angular and fiber components \(X,U\), the
outward second fundamental form is
\begin{equation}\label{eq:southII}
 \II_S(X+U,X+U)=-\mu_S|X|^2-q|U|^2.
\end{equation}
The sign follows from the outward normal \(-\partial_t\).

\subsection{The northern radial profile}
The following direct construction uses a radial profile different from
the implicit cubic profile in the eleven-dimensional paper. Its purpose
is to make all parameter choices and all-plane estimates reproducible.
Let \(\eta:\RR\to[0,1]\) be smooth and nondecreasing, zero on
\(( -\infty,1/4]\), and one on \([1/2,\infty)\). Such a function is
obtained by integrating a nonnegative smooth bump supported in
\((1/4,1/2)\) and dividing by its integral. Set
\(C_\eta=\max|\eta'|\).

For \(z=(a_0,w)\in S(V)\), the infinitesimal action is
\[
 K_z(\xi)=(0,-w\xi),\qquad |K_z\xi|=|w||\xi|\le|\xi|.
\]
We retain a symbol \(L\ge1\) for its uniform bound; \(L=1\) suffices
in the present application. Choose \(\delta>0\) with
\begin{equation}\label{eq:deltachoice}
 \delta\le b/2,\qquad
 \delta\le\frac1{48L^2Ab^3},\qquad
 \delta^3\le\frac1{2L^2Ab\max\{1,C_\eta\}}.
\end{equation}
This choice does not involve \(\eps\). Define
\begin{equation}\label{eq:newprofile}
 F(s)=\frac{s}{(1+s^2/\delta^2)^{1/4}},\qquad s\ge0.
\end{equation}
It is strictly increasing and unbounded. Let \(\ell\) be the unique
number with \(F(\ell)=b\); explicitly
\[
 \ell^2=\frac{b^4/\delta^2+
       \sqrt{b^8/\delta^4+4b^4}}2.
\]
In particular \(\ell\ge b\ge2\delta\).

Put \(d=r_aq\), and set
\begin{equation}\label{eq:northr}
 r(s)=r_a-d\int_s^\ell\eta(v/\delta)\,dv.
\end{equation}
Now choose the fibre size. The southern cap has \(\kappa=1\),
\(\sup|\phi|=A\), and \(\sup|d\phi|=Ab\). Using these data in
\eqref{eq:connection-explicit-threshold}, it is enough to take
\begin{equation}\label{eq:disk-epsilon-choice}
 0<\eps\le\min\left\{\varepsilon_*,\frac1{Ab\ell},
             \left(\frac{e^{-1/4}}{Ab}\right)^{2/3}\right\}.
\end{equation}
Indeed \(\eps A\le1\), so \(q\ell\le1/2\) and
\(d=(Ab/2)\eps^{3/2}e^{\eps A/4}\le1/2<1\).
The southern quotient then has sectional curvature at least \(\eps/4\).
Then
\begin{equation}\label{eq:rcontrol}
 \frac{r_a}{2}\le r\le r_a,\quad
 0\le r'\le d<1,\quad
 r''=\frac d\delta\eta'(s/\delta)\ge0,\quad
 \frac d{r^3}\le\frac{8q}{r_a^2}\le4Ab.
\end{equation}
On the northern source \(D^{15}\times S^3\), use
\begin{equation}\label{eq:northmetric}
 G_N=ds^2+F(s)^2h_{S(V)}+r(s)^2h_{S^3}.
\end{equation}
This is smooth at \(s=0\). In fact \(F\) extends as a smooth odd
function with \(F'(0)=1\), and \(r\) is constant for
\(s\le\delta/4\). The source metric is invariant under \eqref{eq:star}.

\begin{lemma}\label{lem:profile}
For \(0\le s\le\ell\), the profile \eqref{eq:newprofile} satisfies
\[
 FF'\le\delta,\qquad
 \frac{1-F'^2}{F^2}\ge\frac1{3b^2},\qquad
 -\frac{F''}{F}>0,
\]
where the last two expressions are interpreted by their smooth limits
at zero. If \(s\le\delta/2\), then \(-F''/F\ge\delta^{-2}\).
\end{lemma}
\begin{proof}
Put \(x=s/\delta\). Direct differentiation gives
\begin{equation}\label{eq:profilederivatives}
 F'=\frac{1+x^2/2}{(1+x^2)^{5/4}},\qquad
 -\frac{F''}{F}=\frac{6+x^2}{4\delta^2(1+x^2)^2}.
\end{equation}
Also
\[
 FF'=\delta\frac{x}{\sqrt{1+x^2}}
                   \frac{1+x^2/2}{1+x^2}\le\delta.
\]
Since \(F'^2\le(1+x^2)^{-1/2}\), rationalization gives
\[
 \frac{1-F'^2}{F^2}\ge
 \frac1{\delta^2(\sqrt{1+x^2}+1)}.
\]
The denominator equals
\(F^2+\delta^2/\sqrt{1+x^2}+\delta^2\le b^2+2\delta^2\le3b^2\).
Finally \((6+x^2)/[4(1+x^2)^2]\ge1\) when \(0\le x\le1/2\):
the numerator after subtracting the denominator is
\(2-7x^2-4x^4\), decreasing to zero at \(x=1/2\).
Both sectional expressions tend to \(3/(2\delta^2)\) at zero.
\end{proof}

\subsection{Every northern horizontal two-plane}
Positivity of the quotient follows from positivity on source planes
horizontal for the free action. Such a plane can contain radial,
angular, and fiber components simultaneously. We first express the
horizontal condition as a graph relation, then use it to control the
possibly negative mixed terms.

All components in the following calculation are normalized for the
source metric. At fiber coordinate \(u=1\), the vertical space of the
\(\star\)-submersion consists of \((FK_z\xi,r\xi)\) in angular and
fiber components. A horizontal vector therefore has the form
\begin{equation}\label{eq:graph}
 \lambda\partial_s+X+U,\qquad
 U=-\frac F rK_z^*X.
\end{equation}
Right translations give the same statement at every fiber point. The
projection of this horizontal graph onto the base tangent space is
injective. This is the property needed below, not freeness of the
representation on \(S(V)\).

For two horizontal vectors \((\lambda,X,U)\) and \((\mu,Y,W)\), define
\[
 R_0=|\lambda Y-\mu X|^2,\qquad H_0=|X\wedge Y|^2.
\]
The graph relation implies
\begin{align}
 |\lambda W-\mu U|^2&\le L^2F^2r^{-2}R_0,\label{eq:graph1}\\
 |X\otimes W-Y\otimes U|^2&\le2L^2F^2r^{-2}H_0.\label{eq:graph2}
\end{align}
For completeness, \eqref{eq:graph2} is an elementary linear-algebra
inequality. Replace \(Y\) by \(Y-cX\) so that \(X\perp Y\); replace
\(W\) by \(W-cU\) simultaneously. Both tensors and the angular wedge
are unchanged. The mixed squared norm becomes
\(|X|^2|W|^2+|Y|^2|U|^2\), which has the asserted bound. When
\(X,Y\) are dependent, both mixed and angular tensors vanish.

The Koszul formula for \eqref{eq:northmetric} gives five nonzero
sectional blocks: radial--angular, radial--fiber, angular, fiber, and
angular--fiber. There are no cross terms between these blocks. Thus its
sectional numerator on the pair above is exactly
\begin{align}\label{eq:warpedcurvature}
 \mathcal K_N={}&-\frac{F''}{F}R_0
 -\frac{r''}{r}|\lambda W-\mu U|^2
 +\frac{1-F'^2}{F^2}H_0\\
 &+\frac{1-r'^2}{r^2}|U\wedge W|^2
 -\frac{F'r'}{Fr}|X\otimes W-Y\otimes U|^2.\notag
\end{align}
One can obtain these blocks directly from
\(\nabla_X\partial_s=(F'/F)X\),
\(\nabla_U\partial_s=(r'/r)U\), and the unit-sphere curvatures of
the two factors. The negative last term is retained.

\begin{proposition}\label{prop:northpositive}
The quotient of \eqref{eq:northmetric} by \eqref{eq:star} has strictly
positive sectional curvature at every point and on every two-plane.
\end{proposition}
\begin{proof}
By \eqref{eq:rcontrol}, the fiber coefficient in
\eqref{eq:warpedcurvature} is positive. Equations
\eqref{eq:graph1}--\eqref{eq:graph2} give
\[
 \mathcal K_N\ge P_R R_0+P_H H_0,
\quad
 P_R=-\frac{F''}{F}-\frac{L^2F^2r''}{r^3},\qquad
 P_H=\frac{1-F'^2}{F^2}-\frac{2L^2FF'r'}{r^3}.
\]
Lemma~\ref{lem:profile} and \eqref{eq:deltachoice} imply
\[
 P_H\ge\frac1{3b^2}-8L^2Ab\delta\ge\frac1{6b^2}>0.
\]
The radial loss is supported in \(s\le\delta/2\). There \(F\le s\),
so it is at most
\[
 L^2\frac{\delta^2}{4}\frac{4AbC_\eta}{\delta}
 =L^2AbC_\eta\delta\le\frac1{2\delta^2}.
\]
The positive radial term is at least \(1/\delta^2\), leaving
\(P_R\ge1/(2\delta^2)\). Off this support the radial loss vanishes
and \(-F''/F\) remains strictly positive. An independent horizontal
pair projects to an independent base pair; hence \(R_0+H_0>0\).
This proves strict positivity for every horizontal plane at \(s>0\).

At the center, the representation acts trivially on the base point and
freely on the fiber. The horizontal space is the entire base tangent
space. The source is a product near that fiber and the base curvature
there is \(3/(2\delta^2)>0\). Smoothness and the submersion formula
complete the center argument, without a singular polar-coordinate limit
being used as a substitute for smoothness. Finally O'Neill's formula
adds a nonnegative term to the horizontal sectional curvature.
\end{proof}

For a quantitative check, put \(X_*=\ell/\delta\) and
\[
 c_R=\min\left\{\frac1{2\delta^2},
       \frac{6+X_*^2}{4\delta^2(1+X_*^2)^2}\right\},\qquad
 H_*=1+4L^2b^2/r_a^2.
\]
The radial expression in \eqref{eq:profilederivatives} decreases with
\(x\ge0\). The squared norm of a graph vector is at most \(H_*\)
times the squared norm of its base projection; the squared area is
therefore enlarged by at most \(H_*^2\).
Consequently the northern quotient has the positive lower bound
\begin{equation}\label{eq:northlower}
 \secg_N\ge H_*^{-2}\min\{c_R,1/(6b^2)\}.
\end{equation}
No sampling of planes is used to obtain this bound.

\subsection{Boundary identification and smooth gluing}
The two disks are now positively curved up to their boundaries. To apply
Lemma~\ref{lem:harmonic}, it remains to compare their boundary tensors
in the same gauge. This calculation is also where the metric construction
is tied to the exact smooth attaching map.

At \(s=\ell\) the northern data satisfy
\[
 F=b,\quad r=r_a,\quad r'/r_a=q,\quad
 \mu_N=F'(\ell)/b>0.
\]
Both source boundaries, expressed in the common north gauge, have the
same metric \(b^2h_{S(V)}+r_a^2h_{S^3}\) and the same free action.
Their quotient boundary metrics therefore agree under precisely
\(J_\beta\). Both source unit normals are horizontal for the quotient.
Taking their second fundamental forms on a common horizontal lift gives
\[
 \II_N(X+U,X+U)=\mu_N|X|^2+q|U|^2,
\quad
 (J_\beta^*\II_S)(X+U,X+U)=-\mu_S|X|^2-q|U|^2.
\]
Thus the fiber terms cancel and
\begin{equation}\label{eq:seam}
 \II_N+J_\beta^*\II_S=(\mu_N-\mu_S)|X|^2>0
\end{equation}
on every nonzero quotient boundary tangent vector. Indeed the graph
relation forces \(X\ne0\). More explicitly, relative to the common
quotient boundary metric \(h\),
\[
 \II_N+J_\beta^*\II_S\ge
 \frac{\mu_N-\mu_S}{1+L^2b^2/r_a^2}\,h.
\]

The remaining task is to smooth the corner while preserving curvature.
Lemma~\ref{lem:harmonic}, proved in full in
Section~\ref{sec:smoothing}, does this by convolving the inverse
tangential metric in a signed normal collar. The strict inequality
\eqref{eq:seam} supplies its derivative-jump condition. Take the north
disk as the negative side of the signed collar and pull back the south
collar by $J_\beta$. Then $D=2(\II_N+J_\beta^*\II_S)>0$, with exactly
the sign convention used in that lemma. Smoothing changes the metric
inside the collar and preserves the prescribed attaching map.

\begin{theorem}\label{thm:diskconstruction}
Every smooth conjugation-equivariant map \(\beta:S(\RR^3\oplus\HH^3)
\to S^3\) for the representation above determines a smooth positively
curved metric on \(D^{15}_N\cup_{J_\beta}D^{15}_S\).
\end{theorem}
\begin{proof}
The southern connection criterion, Proposition~\ref{prop:northpositive},
the quotient-coordinate calculation, and \eqref{eq:seam} verify every
hypothesis of Lemma~\ref{lem:harmonic}. The order of choices is
\[
 \beta,\chi,\omega\ ;\ A\ ;\ \eta,\delta,\ell\ ;\ \eps
 \ ;\ \text{collar bounds}\ ;\ \zeta.
\]
All connection norms and profile constants are finite when the next
choice is made, so no smallness requirement is circular.
More precisely, define
\[
 \kappa_N=H_*^{-2}\min\{c_R,1/(6b^2)\},\qquad
 \kappa_* = \min\{\eps/4,\kappa_N\}.
\]
Both closed disks have sectional curvature at least \(\kappa_*>0\).
Lemma~\ref{lem:harmonic} gives a smooth metric on the glued sphere
with sectional curvature greater than \(\kappa_*/2\).
\end{proof}
\section{A fixed-choice specification for the first exotic member}
\label{sec:choices}

The preceding arguments allow several auxiliary choices. We now specify
them for $m=1$, so that a fixed smooth gauge family determines a definite
metric by the displayed formulas. The input is a jointly smooth family
$f_1:D^3\to\mathcal G_2$ of evaluation degree $180$, equal to the
identity near $\partial D^3$, as provided by
Lemma~\ref{lem:gauge-existence}. Set $r_0=1/2$ in
\eqref{eq:explicit-beta}, obtaining $\beta_1$, and set
$g=\beta_1^{-1}$. This section specifies the subsequent analytic choices;
it does not supply a closed coordinate formula for $f_1$.

\paragraph{Auxiliary functions.}
Define on $\RR$
\[
 E(s)=\begin{cases}e^{-1/s},&s>0,\\0,&s\leq0,\end{cases}
 \qquad
 \mathcal H(s)=\frac{E(s)}{E(s)+E(1-s)}.
\]
The denominator is positive everywhere. The function $\mathcal H$ is
smooth and nondecreasing, equals zero for $s\leq0$, and equals one for
$s\geq1$. For $0\leq s\leq1$ its denominator is at least $e^{-2}$,
while $0\leq E'(s)\leq4e^{-2}$ wherever $E'$ is nonzero. Differentiating
the quotient therefore gives $0\leq\mathcal H'\leq4$.
With $a=2\pi/3$ and $b=\sqrt3/2$, choose
\[
 \eta(v)=\mathcal H(4v-1),\qquad
 \chi(t)=\mathcal H\left(\frac{3(t-a)}{\pi-a}-1\right).
\]
Then $C_\eta\leq16$, and $\chi$ vanishes for
$t\leq a+(\pi-a)/3$ and equals one for
$t\geq a+2(\pi-a)/3$. These choices define the smooth connection
\eqref{eq:potentials}, hence $C_0,C_1$ and the curvature norm
$\sup|\Omega|$ used in Appendix~\ref{sec:connection-bounds}, before
the fibre size is chosen.

\paragraph{Disk parameters.}
Take $L=1$ and set
\[
 A=4(C_0+2C_1^2+1),\qquad
 \delta=\frac12\min\left\{\frac b2,\frac1{48Ab^3},
                                    \left(\frac1{32Ab}\right)^{1/3}\right\}.
\]
This satisfies \eqref{eq:deltachoice} using $C_\eta\leq16$.
Use the profile $F(s)$ in~\eqref{eq:newprofile} and the unique
$\ell>0$ with $F(\ell)=b$. Evaluate $\varepsilon_*$ from
\eqref{eq:connection-explicit-threshold} with $n=15$, base lower bound
$1$, $\Phi=A$, and $P=Ab$, and choose
\begin{equation}\label{eq:fixed-epsilon}
 \eps=\frac12\min\left\{\varepsilon_*,\frac1{Ab\ell},
                          \left(\frac{e^{-1/4}}{Ab}\right)^{2/3}\right\}.
\end{equation}
Define $r_a,q,d,r_S$ as in Section~\ref{sec:disks}, and write $r_N$
for the function $r$ in \eqref{eq:northr}.
They give $q\ell\leq1/2$, $d\leq1/2$, and the common disk lower bound
\[
 \kappa_*=
 \min\left\{\eps/4,\ H_*^{-2}\min\{c_R,1/(6b^2)\}\right\}>0.
\]

\paragraph{The actual quotient cometrics in a common collar.}
Put $c=\min\{\ell/4,(\pi-a)/12\}$ and let $\tau$ be the signed
normal parameter, with $s=\ell+\tau$ on the north side and
$t=a+\tau$ on the south side. In the common north gauge,
$\chi=0$ throughout this collar. On $Y=S(V)$ let
$\mathcal B=h_{S(V)}^{-1}$ and
\[
 \mathcal T_z=\sum_{\alpha=1}^3K_z(e_\alpha)\otimes K_z(e_\alpha),
\]
where $(e_1,e_2,e_3)$ is any $Q$-orthonormal basis of
$\mathfrak{sp}(1)$. The sum is independent of that basis.
The tangential quotient cometrics, with the southern collar pulled
back by the prescribed $J_{\beta_1}$, are exactly
\begin{align}\label{eq:fixed-cometrics}
 p_-(\tau)&=F(\ell+\tau)^{-2}\mathcal B
                    +r_N(\ell+\tau)^{-2}\mathcal T,&&-c<\tau\leq0,\nonumber\\
 p_+(\tau)&=\sin(a+\tau)^{-2}\mathcal B
                    +r_S(a+\tau)^{-2}\mathcal T,&&0\leq\tau<c.
\end{align}
Here $r_N(\ell+\tau)=r_a+d\tau$, since
$\ell-c\geq3\ell/4\geq3\delta/2$ lies beyond the transition of $\eta$.
To check the formula, at fibre coordinate $u=1$ the quotient map
$(z,u)\mapsto\rho(u^{-1})z$ has differential
$(v,\xi)\mapsto v-K_z\xi$. A covector $\lambda\in T_z^*Y$
therefore pulls back to $(\lambda,-K_z^*\lambda)$, whose squared
dual norm for the source product metric is
\[
 f^{-2}\mathcal B(\lambda,\lambda)
          +r^{-2}\sum_\alpha\lambda(K_z e_\alpha)^2.
\]
Pullback preserves the norm of a quotient covector, proving
\eqref{eq:fixed-cometrics}. There are no radial cross terms. In
particular $\tau$ is a normal geodesic coordinate and
$h_0=(b^{-2}\mathcal B+r_a^{-2}\mathcal T)^{-1}$.

\paragraph{Fixed smoothing choices and scope.}
Define
\[
 \rho(u)=\frac{E(1-u^2)}{\int_{-1}^1E(1-v^2)\,dv},\qquad
 \varrho(u)=2\rho(2u),\qquad
 \theta(\tau)=1-\mathcal H\left(\frac{9\tau^2/c^2-1}{3}\right).
\]
The dilation makes $\varrho\in C_c^\infty((-1,1))$, as required in
\eqref{eq:harmonic-definition}. It is even, nonnegative, and has integral
one. The cutoff $\theta$ is one for $|\tau|\leq c/3$ and zero for
$|\tau|\geq2c/3$. Compute $M,\mu,\Theta,C_p,B_{\mathrm{sec}}$ from
\eqref{eq:fixed-cometrics} by the fixed-jet budget in
Appendix~\ref{sec:smoothing-bounds}, and choose
\[
 \zeta=\frac12\min\left\{\frac c{12},\frac1{32M},\frac1{C_p},
        \frac{\kappa_*}{2000M^3},
        \frac{\kappa_*}{2B_{\mathrm{sec}}C_p}\right\}.
\]
Formula \eqref{eq:harmonic-definition} now specifies the final collar
metric and leaves the two disk metrics unchanged outside that collar.
Its sectional curvature is greater than $\kappa_*/2$ everywhere.

For this member $k=180$, the previously proved topological formulas give
\[
 \bigl(\langle p_1^4\rangle,\langle p_1^2p_2\rangle,
        \langle p_2^2\rangle,\langle p_1p_3\rangle\bigr)
       =(46080,17280,0,10080),\qquad
 \mu(\Sigma_1)=15/1016\pmod\ZZ.
\]
Thus the smooth type is nonstandard by
Theorem~\ref{thm:independent-topology}. The analytic prescription uses
the exact uniform norms defined above. No numerical values of those
connection or collar norms have been certified here, and no numerical
implementation of the resulting metric is claimed. The recipe fixes
the auxiliary choices after $f_1$ is given; it does not select a
canonical gauge family or a canonical metric on the smooth type.
\section{Completion of the proof}

\begin{proof}[Proof of Theorem~\ref{thm:main}]
Theorem~\ref{thm:independent-topology} supplies the specified smooth
conjugation-equivariant map \(\beta_m\), its exact disk presentation,
and the value of \(\mu(\Sigma_m)\) for each integer \(m\).
Theorem~\ref{thm:diskconstruction}, proved by the metrics and estimates
in this paper, gives a smooth positively curved metric on that
same marked disk gluing. The invariant and orientation argument in
Theorem~\ref{thm:independent-topology} give precisely the two asserted
counts.
\end{proof}

The small parameters depend on the chosen smooth gauge family and
connection. This is sufficient for existence on each fixed sphere.
There is no claim of a lower bound independent of \(m\), a preferred
metric for each smooth type, or a classification of all positively
curved fifteen-spheres. Similarly, the proof does not construct the
additional non-isometric circle realization studied in the
eleven-dimensional paper; that property is not needed here.

The auxiliary estimates in the appendices make each smallness choice
traceable to fixed data. Obtaining a closed formula for a nontrivial
gauge family, and then certified numerical bounds for its connection,
would turn the specification in Section~\ref{sec:choices} into a more
explicit example. This further computation is separate from the
existence argument above.

\subsection*{Use of computational tools}
Generative AI tools assisted with drafting, symbolic calculations, and
checks of intermediate arguments in preparing this version. Exact-arithmetic
checks cover selected algebraic identities and invariant counts; they do
not constitute a formal verification of the geometric and topological
proofs.
\appendix
\section{Quantitative bounds for the connection metric}\label{sec:connection-bounds}

Here is a fully specified version of the small-parameter choice in
Proposition~\ref{prop:independent-connection}.  Write
\[
 n=\dim B,\qquad N=\binom n2,\qquad m=3n,
 \qquad F=\sup_B|\Omega|,\quad
 \Phi=\sup_B|\phi|,\quad P=\sup_B|d\phi|.
\]
The norm of $\Omega$ is the Hilbert--Schmidt norm as a section of
$\Lambda^2T^*B\otimes\mathfrak{sp}(1)$; thus every
$|\Omega(e_i,e_j)|$ is at most $F$.  Keep the operator-norm constants
$C_0,C_1$ in the proposition.  If only Hilbert--Schmidt bounds are
available, the sufficient replacements
\[
 C_0\leq F^2/2,\qquad C_1\leq\tfrac12\sup_B|D\Omega|
\]
follow by Cauchy--Schwarz.  In the first inequality one uses
$\sum_{i,k,u}|\Omega(e_i,e_k)_u|^2=2|\Omega|^2$ in the formula for $S$.

Let $\mathcal E_\varepsilon$ be the difference between the congruent
modified curvature form in \eqref{eq:connection-limit} and its displayed
limit.  Label its three spaces $H=\Lambda^2TB$, $M=TB\otimes
\mathfrak{sp}(1)$ and $V=\Lambda^2\mathfrak{sp}(1)$; their dimensions
are $N,m,3$.  For
\[
 0<\varepsilon\leq1,\qquad \varepsilon\Phi\leq1,
\]
the following are bounds for the \emph{operator norms} of all six blocks:
\[
\begin{array}{c|c|l}
\text{block}&\text{bound}&\text{specified constant}\\ \hline
 HH&a\varepsilon&a=eNF^2\\
 HM&b\varepsilon&b=\sqrt e\bigl(\Phi C_1/2+\sqrt{Nm}\,PF\bigr)\\
 HV&c\sqrt\varepsilon&c=\sqrt{3N}\bigl(3F/2+enF^2/2\bigr)\\
 MM&d\varepsilon&d=e\Phi C_0+P^2/4\\
 MV&f\varepsilon^{3/2}&f=\sqrt{3m}\,\sqrt e\,PF/2\\
 VV&j\varepsilon&j=e\Phi+P^2/4.
\end{array}
\]
These constants require no bound on $R_B$ and no additional derivative
of $\phi$: its Hessian occurs exactly in the limiting mixed block.

To verify the table, use $r^2\leq e\varepsilon$,
$|d\log r|\leq\varepsilon P/2$,
$|e^{\varepsilon\phi}-1|\leq e\varepsilon\Phi$, and
$|e^{\varepsilon\phi/2}-1|\leq\sqrt e\varepsilon\Phi/2$.
The horizontal error has entries at most $r^2F^2$.  The horizontal--mixed
derivative error is
$(e^{\varepsilon\phi/2}-1)L$; the remaining entries have magnitude at
most $\sqrt e\varepsilon PF$.  After adding the four-form, the
horizontal--vertical entries before the $\sqrt\varepsilon$ congruence
are at most $3F/2+r^2nF^2/2$.  The mixed error is exactly
\[
 (e^{\varepsilon\phi}-1)S
       -\tfrac\varepsilon4(d\phi\otimes d\phi)\otimes Q.
\]
The mixed--vertical entries are bounded by
$r|d\log r|F\leq\sqrt e\varepsilon^{3/2}PF/2$.
Finally the vertical error is the scalar
$e^{-\varepsilon\phi}-1-\varepsilon^3|d\phi|^2/4$.
For an $r\times s$ matrix with entries bounded by $b$, its operator
norm is at most $\sqrt{rs}\,b$; this accounts for every dimension
factor in the table.

Put
\[
 \alpha=\min\{\kappa/2,1\},\qquad
 K=\max\{a+b+c,\ b+d+f,\ c+f+j\}.
\]
The table and the maximum-row-sum bound for a symmetric three-by-three
matrix give
$\|\mathcal E_\varepsilon\|_{\rm op}\leq K\sqrt\varepsilon$.
Thus one completely explicit sufficient threshold is
\begin{equation}\label{eq:connection-explicit-threshold}
 \varepsilon_*=
 \min\left\{1,\frac1{\max\{1,\Phi\}},
       \left(\frac{\alpha}{2\max\{1,K\}}\right)^2\right\}.
\end{equation}
For $0<\varepsilon\leq\varepsilon_*$ the proof of the proposition
gives not only positivity but the quantitative lower bound
\begin{equation}\label{eq:connection-explicit-sec}
 \sec_{G_\varepsilon}\geq\tfrac12\alpha\varepsilon.
\end{equation}
Indeed, for a unit simple bivector the weighted square norm in that proof
is $|\xi_H|^2+\varepsilon|\xi_M|^2+
\varepsilon^{-1}|\xi_V|^2\geq\varepsilon$.
The four-form makes no contribution to its sectional numerator.

For a less conservative numerical choice, one may retain the six powers
of $\varepsilon$ in the table and require all row sums of
\[
 \begin{pmatrix}
 a\varepsilon&b\varepsilon&c\sqrt\varepsilon\\
 b\varepsilon&d\varepsilon&f\varepsilon^{3/2}\\
 c\sqrt\varepsilon&f\varepsilon^{3/2}&j\varepsilon
 \end{pmatrix}
\]
to be at most $\alpha/2$.  These are three explicit monotone polynomial
inequalities in $\sqrt\varepsilon$ and can be certified without sampling
points or two-planes.  Formula \eqref{eq:connection-explicit-threshold}
is available when such refinement is unnecessary.
\section{Quantitative bounds for harmonic smoothing}\label{sec:smoothing-bounds}
All constants in this paragraph are computed after the two metrics and
the collar have been fixed. Set $g_0=dt^2+h_0$ and identify a tangential
cometric with an endomorphism using $h_0$. Norms of endomorphisms are
operator norms; norms of derivative tensors are induced multilinear
norms. The norm of a jet is the maximum of its indicated block norms.
On the two closed half-collars $|t|\le3c/4$, put
\[
 M=\max\left\{1,\|R^{h_0}\|_{(1,3)},
   \max_{i+j\le3}\sup
       \| (\nabla^0)^i\partial_t^j p\|\right\}.
\]
\[
 \mu=\min\lambda_{h_0}(p)>0,\quad
 \Theta=\max\{1,\|\theta'\|_\infty,\|\theta''\|_\infty\}.
\]
The derivatives and minimum are taken separately on the two sides,
including their one-sided values at zero. In particular $\mu\le1$,
since $p(0)=I$. Let $r=2M\zeta$ and assume $\zeta<c/12$ and $r\le1/4$.
In the seam strip the following bounds suffice:
\[
\begin{array}{c|c|c}
\text{quantity}&\text{bound}&\text{highest required derivative of }p\\ \hline
 E_0\text{ in }C^2(Y)&r&(\nabla^0)^2\partial_t p\\
 E_1\text{ in }C^1(Y)&r&\nabla^0\partial_t^2p\\
 E_2\text{ in }C^0(Y)&r&\partial_t^3p\\
 \mathcal A_H-\overline{\mathcal A}
       -\tfrac34vD^2-\tfrac12JHDH&60M^2r&\text{total order }3\\
 \mathcal B_H-\overline{\mathcal B}
       -\tfrac14v\Lambda^2D&160M^2r&\text{total order }3\\
 \mathcal C_H-\overline{\mathcal C}&90Mr&\text{total order }3
\end{array}
\]
The tangential row is evaluated on unit simple tangential bivectors;
the mixed row on a unit vector and a unit simple tangential bivector.
The other rows use their stated operator or multilinear norms. Each
of the first three bounds is separate; their sum is at most $3r$.

Here is a direct verification of the budgets. In the seam strip,
\[
 \|p_\zeta-I\|\le r,\quad \|H\|\le4/3,\quad
 \|H-I\|\le4r/3,\quad
 \|\nabla^0 H\|\le4r,\quad
 \|(\nabla^0)^2H\|\le8r.
\]
Write $A=p'_\zeta$, $A_0=-\overline P$, and let
$B=p''_{\zeta,\mathrm{reg}}$, $B_0=2\overline{P^2}-\overline Q$.
Then $\|A\|,\|A_0\|,\|B\|,\|B_0\|\le M$ and
$\|A-A_0\|,\|B-B_0\|\le r$. Telescoping products gives
\[
 \|2HAHAH-2A_0^2\|<24M^2r,\qquad
 \|HBH-B_0\|<6M^2r,
\]
and
\[
 \|H'-\overline P\|\le8Mr,\quad \|H'\|<2M,\quad
 \|H'p_\zeta H'-(\overline P)^2\|\le30M^2r.
\]
Thus the radial regular error is less than $60M^2r$.
The atom $JHDH$ is excluded from these telescoping estimates and is
retained exactly.

For $C=\nabla^H-\nabla^0$, the connection difference formula gives
\[
 \|C\|<8r,\quad \|\nabla^0 C\|<17r,\quad
 \|R^H-R^{h_0}\|_{(1,3)}\le66r.
\]
Lowering an index with $H$ and changing that lowering metric contributes
at most $(132+2M)r$ to the intrinsic $(0,4)$ error. The change in the
Gauss quadratic expression is at most $20M^2r$, so the tangential
bound is at most $160M^2r$. Differentiating $H'=-Hp'_\zeta H$ gives
\[
 \|\nabla^0(H'-\overline P)\|\le50Mr.
\]
The two covariant slots in the Codazzi connection correction cost less
than $32Mr$. Thus its error is less than $90Mr$.
For $|A|^2+|B|^2=1$, the combined regular error is at most
\[
 60M^2r|A|^2+160M^2r|B|^2+180Mr|A||B|
 \le250M^2r=500M^3\zeta.
\]
Consequently the seam estimate can be written explicitly as
\[
 \secg_{g_\zeta}\ge
 \frac{\kappa-500M^3\zeta}{(1+4M\zeta)^2},\qquad |t|\le\zeta.
\]

For the outer strip, define $C_p=4\Theta M$. The product rule and
one-sided Taylor estimates give
\[
 \max_{i+j\le2}\| (\nabla^0)^i\partial_t^j(p_\zeta-p)\|
 \le C_p\zeta,\qquad \zeta\le|t|\le2c/3.
\]
Indeed the three possible normal-derivative coefficients are
$1$, $1+\Theta$, and $1+3\Theta$, all at most $4\Theta$.
No fourth derivative is needed. If $C_p\zeta\le1$, inverse
differentiation and telescoping give the explicit metric-jet bound
\[
 \|H-h\|_{C^2}\le C_I C_p\zeta,\qquad
 C_I=6\mu^{-4}(M+1)^2+6\mu^{-3}(M+1)+\mu^{-2}.
\]
Both cometrics have eigenvalues in $[\mu,M]$. The same interval
contains the zeroth-order part of every segment between their jets.

For a finite, reconstructible curvature constant, define $\mathcal E$
to be the bundle of tangential cometric jets of order two satisfying
\[
 \mu I\le P\le MI,\qquad
 \max_{1\le i+j\le2}\|P_{ij}\|\le M+1.
\]
It lies over the fixed compact half-collars. Include arbitrary formal
derivative tensors with these bounds; imposing jet compatibility is
unnecessary. Use the product connection to convert the derivative
blocks into full jets. For a $g_0$-orthonormal pair $u,v$, let
$\mathcal S_y(j;u,v)$ be the sectional-curvature rational expression
of $dt^2+P^{-1}$, evaluated from the jet $j$. This expression is
reconstructed from
\[
 H_{;a}=-HP_{;a}H,\qquad
 H_{;ab}=HP_{;b}HP_{;a}H+HP_{;a}HP_{;b}H-HP_{;ab}H,
\]
the connection difference formula
\[
 2g(C(X,Y),Z)=(\nabla^0_Xg)(Y,Z)+(\nabla^0_Yg)(X,Z)
                              -(\nabla^0_Zg)(X,Y),
\]
and the $(1,3)$-tensor identity
$R^g=R^{g_0}+(\nabla^0 C)_{\mathrm{alt}}+(C*C)_{\mathrm{alt}}$,
followed by lowering an index and dividing by
$g(u,u)g(v,v)-g(u,v)^2$. Define
\[
 B_{\mathrm{sec}}=
 \max\left\{1,\sup_{y,j\in\mathcal E,\ (u,v)\in\mathrm{St}_2(g_0)}
             \|D_j\mathcal S_y(j;u,v)\|\right\}.
\]
This is a derivative norm of an explicitly specified rational map on
a specified compact set, not a curvature-positivity condition.
All metric eigenvalues are at least $M^{-1}$, so its area denominator
is at least $M^{-2}$ and the constant is finite. The derivative norm
uses the stated maximum jet norm, including its dimension-dependent
conversion to a full tensor jet. It can be bounded by differentiating
the displayed formulas on this compact set. The mean value theorem
along the straight jet segment now gives
\[
 \secg_{g_\zeta}\ge\kappa-B_{\mathrm{sec}}C_p\zeta
 \quad\text{on the outer strip}.
\]
In particular, one explicit sufficient threshold for this proof is
\[
 0<\zeta<\min\left\{
 \frac c{12},\frac1{32M},\frac1{C_p},
 \frac{\kappa}{2000M^3},
 \frac{\kappa}{2B_{\mathrm{sec}}C_p}\right\}.
\]
Every quantity on the right is defined by the fixed one-sided data,
the fixed cutoff, and the already established one-sided lower bound
$\kappa$. None depends on the smoothed metric or on a positivity
claim to be proved about that metric.

\end{document}